\documentclass[11pt,a4paper]{article}

\usepackage{amsthm}	
\usepackage{amsfonts, amssymb, amsmath, amscd, latexsym}
\usepackage{eucal, enumerate, latexsym}
\usepackage{graphics}
\usepackage{graphicx}
\usepackage{mathscinet}
\usepackage{psfrag}
\usepackage{tikz}
\usetikzlibrary{decorations.markings}

\usepackage{hyperref}

\newcommand{\R}{\mathbb R}

\newtheorem{theorem}{Theorem}[section]

\newtheorem{proposition}{Proposition}[section]
\newtheorem{definition}{Definition}[section]
\newtheorem{remark}{Remark}[section]

\newtheorem{example}{Example}[section]

\newcommand{\supp}{\mathop{\rm supp}}

\renewcommand{\epsilon}{\varepsilon}
\newcommand{\eps}{\epsilon}
\renewcommand{\phi}{\varphi}
\newenvironment{proofof}[1]{\smallskip\noindent\emph{Proof of #1.}%
\hspace{1pt}}{\hspace{-5pt}{\nobreak\quad\nobreak\hfill\nobreak%
$\square$\vspace{8pt}\par}\smallskip\goodbreak}

\makeatletter

\newlength{\captionwidth}
\long\def\@makecaption#1#2{%
   \vskip 10\p@
   \setbox\@tempboxa\hbox{#1: #2}%
   \ifdim \wd\@tempboxa > \captionwidth 
       \hbox to\hsize{\hfil
       \parbox[t]{\captionwidth}{
       \small#1: \small#2\par}
       \hfil}
     \else
       \hbox to\hsize{\hfil\box\@tempboxa\hfil}%
   \fi}

\makeatother

\usetikzlibrary{decorations.pathreplacing,decorations.markings}
\tikzset{
  on each segment/.style={
    decorate,
    decoration={
      show path construction,
      moveto code={},
      lineto code={
        \path [#1]
        (\tikzinputsegmentfirst) -- (\tikzinputsegmentlast);
      },
      curveto code={
        \path [#1] (\tikzinputsegmentfirst)
        .. controls
        (\tikzinputsegmentsupporta) and (\tikzinputsegmentsupportb)
        ..
        (\tikzinputsegmentlast);
      },
      closepath code={
        \path [#1]
        (\tikzinputsegmentfirst) -- (\tikzinputsegmentlast);
      },
    },
  },
  mid arrow/.style={postaction={decorate,decoration={
        markings,
        mark=at position .6 with {\arrow[#1]{stealth}}
      }}},
}

\begin{document}

\title{Wavefronts for degenerate diffusion-convection reaction equations
with sign-changing diffusivity}

\author{
Diego Berti\footnote{Department of Sciences and Methods for Engineering, University of Modena and Reggio Emilia, Italy}
\and
Andrea Corli\footnote{Department of Mathematics and Computer Science, University of Ferrara, Italy}
\and
Luisa Malaguti\footnotemark[1]
}



\maketitle

\begin{abstract}
We consider in this paper a diffusion-convection reaction equation in one space dimension. The main assumptions are about the reaction term, which is monostable, and the diffusivity, which changes sign once or even more than once; then, we deal with a forward-backward parabolic equation. Our main results concern the existence of globally defined traveling waves, which connect two equilibria and cross both regions where the diffusivity is positive and regions where it is negative. We also investigate the monotony of the profiles and show the appearance of sharp behaviors at the points where the diffusivity degenerates. In particular, if such points are interior points, then the sharp behaviors are new and unusual.

\vspace{1cm}
\noindent \textbf{AMS Subject Classification:} 35K65; 35C07, 34B40, 35K57

\smallskip
\noindent
\textbf{Keywords:} Sign-changing diffusivity, traveling-wave solutions, diffusion-convection reaction equations, sharp profiles.
\end{abstract}

\section{Introduction}\label{s:I}

This paper deals with traveling-wave solutions to degenerate parabolic equations of forward-backward type. More precisely, we consider the equation
\begin{equation}\label{e:E}
\rho_t + f(\rho)_x=\left(D(\rho)\rho_x\right)_x + g(\rho), \qquad t\ge 0, \, x\in \R.
\end{equation}
We denote with $\rho=\rho(t,x)$ the unknown function; also in view of applications we understand $\rho$ as a (normalized) density or a concentration and then assume that $\rho$ is valued in the interval $[0,1]$. On the convective term $f$ we only assume
\begin{itemize}
\item[{(f)}]\, $f\in C^1[0,1]$, $f(0)=0$.
\end{itemize}
The condition $f(0)=0$ just fixes a flux representative, since convection is only defined up to an additive constant. In the following, for brevity, we denote the derivative of $f$ as $h(\rho)=\dot f(\rho)$. 
We assume that $g$ satisfies
\begin{itemize}
\item[{(g)}]\, $g\in C^0[0,1]$, $g>0$ in $(0,1)$, $g(0)=g(1)=0$.
\end{itemize}
Condition (g) is natural in this framework and then $g$ is said of {\em monostable} type. The diffusivity $D$ is required to satisfy one of the following assumptions, for some $\alpha, \beta \in (0,1)$:
\begin{itemize}
\item[{$(\rm D_{pn})$}] \, $D\in C^1[0,1]$, $D>0$ in $(0,\alpha)$ and $D<0$ in $(\alpha,1)$;

\item[{$(\rm D_{np})$}] \, $D\in C^1[0,1]$, $D<0$ in $(0,\beta)$ and $D>0$ in $(\beta,1)$.

%
\end{itemize}
Each condition is labelled following the sign of $D$: \lq\lq{}pn\rq\rq{} in {{$\rm (D_{pn})$} means that $D$ is first positive and then negative, and the other so on. Above, $\alpha$ denotes a zero of $D$ such that $\dot{D}(\alpha)\le 0$ while $\beta$ denotes a zero of $D$ such that $\dot{D}(\beta)\ge 0$.

\begin{figure}[htb]
\begin{center}

\begin{tikzpicture}[>=stealth, scale=0.5]
\draw[->] (0,0) node[below]{\footnotesize{$0$}} --  (6,0) node[below]{$\rho$} coordinate (x axis);
\draw[->] (0,0) -- (0,3) node[right]{$f$} coordinate (y axis);
\draw[thick] (0,0) .. controls (1.5,4) and (3.5,4) .. (5,1) ;
\draw[dotted] (5,0) -- (5,1);
\draw(5,0)  node[below]{\footnotesize{$1$}};

\begin{scope}[xshift=8cm]
\draw[->] (0,0) --  (6,0) node[below]{$\rho$} coordinate (x axis);
\draw[->] (0,0) -- (0,3) node[right]{$g$} coordinate (y axis);
\draw[thick] (0,0) .. controls (1.6,4) and (4.2,3.4) .. (5,0)  node[below]{\footnotesize{$1$}};
\end{scope}

\begin{scope}[xshift=16cm]
\draw[->] (0,0) --  (6,0) node[below]{$\rho$} coordinate (x axis);
\draw[->] (0,0) -- (0,3) node[right]{$D$} coordinate (y axis);
\draw (0,0) -- (0,-1.5);
\draw[thick] (0,1) .. controls (1,3) and (2,3) .. (3.2,0) node[right=5, above=0]{\footnotesize{$\alpha$}};
\draw[thick] (3.2,0) .. controls (4,-2) and (4.5,-2) .. (5,-1);
\draw[dotted] (5,-1) -- (5,0) node[above]{\footnotesize{$1$}};
\draw(3,-3) node[above]{$(\rm D_{pn})$};
\end{scope}

\begin{scope}[xshift=24cm]
\draw[->] (0,0) --  (6,0) node[below]{$\rho$} coordinate (x axis);
\draw[->] (0,0) -- (0,3) node[right]{$D$} coordinate (y axis);
\draw (0,0) -- (0,-1.5);
\draw[thick] (0,-1) .. controls (1,-3) and (2,-3) .. (3.2,0) node[right=5, below=0]{\footnotesize{$\beta$}};
\draw[thick] (3.2,0) .. controls (4,2) and (4.5,2) .. (5,1);
\draw[dotted] (5,1) -- (5,0) node[below]{\footnotesize{$1$}};
\draw(3,-3) node[above]{$(\rm D_{np})$};
\end{scope}

\end{tikzpicture}

\end{center}
\caption{\label{f:f}{Typical plots of the functions $f$, $g$ and $D$.}}
\end{figure}
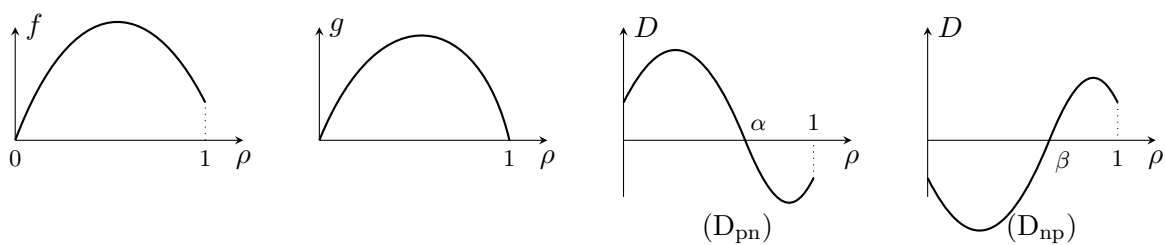

The above assumptions on $D$ are the main issue of this paper, and make \eqref{e:E} a {\em forward} parabolic equation where $D>0$ and a {\em backward} parabolic equation where $D<0$.
The above conditions leave also open the possibility that $D$ vanishes at $0$ or $1$.
%
%
%
%
%

There are several motivations to study forward-backward parabolic equations as \eqref{e:E}: for a short list of different modeling we quote \cite{HPO, Maini-Malaguti-Marcelli-Matucci2006, Padron} for biology, \cite{DJLW} for geophysics, \cite{Kerner-Osipov} for thermodynamics. However, our main source of inspiration has been the recent modeling of collective movements, namely of vehicular flows and crowds dynamics.  About this topic, we refer to \cite{Garavello-Han-Piccoli_book, Garavello-Piccoli_book, Rosinibook} for general information, to \cite{Bellomo-Delitala-Coscia, Bellomo-Dogbe, BTTV} for the modeling using degenerate parabolic equations, to \cite{CdRMR} for the study of star graphs, while for sign-changing diffusivities we refer to \cite{CM-ZAMP, CM-Indam, Nelson_2000} and the references in the two former papers.

Our interest in this paper is about {\em traveling-wave} solutions (TWs for short) to \eqref{e:E}; they are solutions to \eqref{e:E} of the form $\rho(t,x)= \phi(x-ct)$. Here, the function $\phi=\phi(\xi)$ is the {\em profile} of the TW and the real number $c$ is its {\em speed}. The equation for the profiles is then
\begin{equation}\label{e:ODE}
\left(D(\phi)\phi^{\prime}\right)^{\prime}+\left(c -h(\phi)\right)\phi' + g(\phi)=0,
\end{equation}
where ${'}$ denotes the derivative with respect to $\xi$. More precisely we focus on {\em wavefronts}, i.e., globally-defined, nonconstant and monotone profiles. To fix ideas we deal with {\em non-increasing} profiles and this leads us, because of (g), to impose the conditions
\begin{equation}\label{e:infty}
\phi(-\infty)=1, \qquad \phi(+\infty)=0.
\end{equation}
The study of non-decreasing profiles, in which the conditions in \eqref{e:infty} are switched, is not explicitly treated in this paper. Nevertheless, all the results can be rephrased in that case, once that (roughly speaking) the direction of speeds is reversed. Clearly, even under \eqref{e:infty} the solutions to \eqref{e:ODE} are at most unique up to horizontal shifts. An interesting issue is whether the equilibria $0$ and $1$ can be reached by a wavefront $\phi$ at a finite value $\xi_0$. This possibly occurs if $D$ vanishes at those points, and in such cases $\phi$ is necessarily constant on either $(\xi_0,\infty)$ or $(-\infty,\xi_0)$, with values $0$ and $1$, respectively. The profile is called {\em sharp} if it is not differentiable at $\xi_0$. In this case $\phi$ is not a classical solution of \eqref{e:ODE} (see Definition \ref{d:tws}). We refer to \cite{GK} for more information on traveling waves. The case where $D$ changes sign has been considered by several authors {\em but only when $f=0$}. About this case, we quote \cite{Bao-Zhou2014, Maini-Malaguti-Marcelli-Matucci2006} for $D$ satisfying {$(\rm D_{np})$} and {$(\rm D_{pn})$}, respectively, and monostable $g$; \cite{Maini-Malaguti-Marcelli-Matucci2007} for the case {$(\rm D_{pn})$} and bistable $g$ (i.e., $g$ changes sign once).

The main result of the current paper is that there still exist wavefronts joining $1$ with $0$, which travel across regions where $D$ changes sign. In our approach the profiles are constructed by suitably pasting  two {\em semi-wavefronts} and possibly a \emph{traveling wave solution} in a bounded interval (see the next section for a definition), as in \cite{CM-ZAMP}. As a consequence of this procedure, one realizes that the assumptions on $D$ can be somewhat relaxed, as in \cite{Bao-Zhou2014}. Indeed, for instance in case $(\rm D_{pn})$, it is sufficient to require $D\in C^0[0,1]$ and $D\in C^1[0,\alpha]\cap C^1[\alpha,1]$: the derivatives at $\alpha^-$ and $\alpha^+$ may be different.

We prove that the wavefronts constructed in this way are unique (up to shifts) and provide results about the strict monotony of profiles; in particular, we characterize when they are sharp. At last, we give rather precise bounds on the critical thresholds by exploiting the estimates obtained in \cite{BCM1}; there, in turn, we used some related recent results proved in \cite{Marcelli-Papalini}.
The main tool underlying our approach is a well-known reduction (in regions where $D$ has constant sign) of Equation \eqref{e:ODE} to singular first-order systems \cite{MMconv} and its study by comparison-type techniques. More precisely, if we denote $z(\phi):=D(\phi)\phi'$, where $\phi'$ is computed at $\phi^{-1}(\phi)$ (notice that the inverse function of $\phi$ exists because of the monotony of $\phi$), then we are reduced to consider the problems, for instance in the case $(\rm D_{pn})$,
\begin{equation}
\label{e:zIntro}
\begin{cases}
\dot{z}(\varphi)=h(\varphi)-c-\frac{\left(D g \right)(\varphi)}{z(\varphi)}, \ &\varphi\in (0,\alpha),\\
z(\varphi) < 0 , \ &\varphi \in (0,\alpha),\\
z(0)=0,
\end{cases}
\hbox{ or }
\begin{cases}
\dot{z}(\varphi)=h(\varphi)-c-\frac{\left(D g \right)(\varphi)}{z(\varphi)}, \ &\varphi\in (\alpha,1),\\
z(\varphi) > 0 , \ &\varphi \in (\alpha,1),\\
z(1)=0,
\end{cases}
\end{equation}
and similar problems in the other cases. We refer to \cite{BCM1} for a detailed study of $\eqref{e:zIntro}_1$.

Problems \eqref{e:zIntro} seems to suggest that the roles played by $D$ and $g$ are interchangeable; this is not true, in general. Consider for instance the bistable (or Allen-Cahn) equation, where $D>0$ in $(0,1)$ but $g$ satisfies $g(0)=g(\alpha)=g(1)=0$, $g<0$ in $(0,\alpha)$ and $g>0$ in $(\alpha,1)$. In this case it is known \cite{Aronson-Weinberger}, at least if $f=0$, that equation \eqref{e:E} admits a {\em unique} admissible speed $c$ corresponding to a wavefront from $1$ to $0$. In the current case $(\rm D_{np})$, on the contrary, where the role of $D$ and $g$ is commuted with respect to the bistable case, we shall find a whole half line of admissible speeds. The same result holds for the case $(\rm D_{pn})$.

Here follows an outline of the paper. Section \ref{sec:main} contains the statements of our main results. In Section \ref{sec:dpn}, we first briefly recall some preliminary results in order to keep the paper self-contained; then we prove the main result under $(\rm D_{pn})$. Section \ref{sec:dnp} deals with the case $(\rm D_{np})$. In Section \ref{sec:examples} we give some explicit examples; they aim at showing the role played by the convection term $f$, the qualitative difference of the thresholds in the cases $(\rm D_{pn})$ and $(\rm D_{np})$, the occurrence of non-regular profiles. At last, Section \ref{sec:two-sign} extends the previous results to the case when $D$ changes sign twice.

\section{Main results}\label{sec:main}
\setcounter{equation}{0}

As we mentioned in the Introduction, traveling-waves can fail to be of class $C^1$ in the whole of their domain. The following definition makes precise what we mean by a TWs, see \cite{GK}.

\begin{definition}\label{d:tws} Assume $f,D,g\in C^0[0,1]$ and let $I\subset\R$ be an open interval. Let $\varphi\in C^0(I)$ be a function valued in $[0,1]$, which is differentiable a.e. and such that $D(\varphi) \varphi^{\, \prime}\in L_{\rm loc}^1(I)$; at last, let $c$ be a real constant.

Then the function $\rho(x,t):=\varphi(x-ct)$, for $(x,t)$ with $x-ct \in I$, is a {\em traveling-wave} solution (briefly, a TW) to equation \eqref{e:E} with wave speed $c$ and wave profile $\phi$ if we have
\begin{equation}\label{e:def-tw}
\int_I \left(D\left(\phi(\xi)\right)\phi'(\xi) - f\left(\phi(\xi)\right) + c\phi(\xi) \right)\psi'(\xi) - g\left(\phi(\xi)\right)\psi(\xi)\,d\xi =0,
\end{equation}
for every $\psi\in C_0^\infty(I)$.
\end{definition}

We now give some definitions about TWs. A TW is: {\em global} if $I=\R$; {\em strict} if $I\ne \R$ and $\phi$ cannot be extended to $\R$; {\em classical} if $\varphi$ is differentiable, $D(\varphi) \varphi'$ is absolutely continuous and \eqref{e:ODE} holds a.e.; {\em sharp at $\ell$} if there exists $\xi_{\ell}\in I$, with $\phi(\xi_{\ell})=\ell$, such that $\phi$ is classical in $I\setminus\{\xi_\ell\}$ and not differentiable at $\xi_{\ell}$. Analogously, a TW is {\em classical} at $\ell$ if it is classical in a neighborhood of $\xi_\ell$.
\par
Moreover, a TW is: {\em a wavefront} if it is global, with a monotonic, non-constant profile $\phi$ which satisfies either \eqref{e:infty} or the converse condition; {\em a semi-wavefront to} $1$ (or {\em to} $0$) if $I=(a,\infty)$ for $a\in\R$, the profile $\phi$ is monotonic, non-constant and $\phi(\xi)\to1$ (respectively, $\phi(\xi)\to0$) as $\xi\to\infty$; {\em a semi-wavefront from} $1$ (or {\em from} $0$) if $I=(-\infty,b)$ for $b\in\R$, the profile $\phi$ is monotonic, non-constant and $\phi(\xi)\to1$ (respectively, $\phi(\xi)\to0$) as $\xi\to-\infty$. About semi-wavefronts, we say that $\phi$ connects $\phi(a^+)$ ($1$ or $0$) with $1$ or $0$ (resp., with $\phi(b^-)$).

The problem of the loss of regularity of $\phi$ depends on whether the parabolic equation degenerates or not; more precisely, by arguing on the very equation \eqref{e:ODE}, it is easy to see that if $f, D$ are of class $C^1$ and $g$ of class $C^0$, then $\phi$ is classical in every interval $I_{\pm}\subseteq I$ where $\pm D\left(\phi(\xi)\right) > 0$ for $\xi \in I_{\pm}$; moreover, $\phi\in C^2(I_\pm)$ (see e.g. \cite[Lemma 2.20]{GK}).

We recall that for a function $q:[0,1]\to \R$, the notation $D_{+}q(\rho_0)$ and $D_{-}q(\rho_0)$, with $\rho_0\in[0,1]$, stands for the {\em right}, resp., {\em left lower Dini-derivative} of $q$ at $\rho_0$; analogously, $D^{\pm}(q)$ represent the {\em right} and {\em left upper Dini-derivatives} of $q$.
More explicitly,
\[
D_{\pm}q(\rho_0):=\liminf_{\rho \to \rho_0^\pm}\frac{q(\rho)-q(\rho_0)}{\rho-\rho_0},
\quad
D^{\pm}q(\rho_0):=\limsup_{\rho \to \rho_0^\pm}\frac{q(\rho)-q(\rho_0)}{\rho-\rho_0}.
\]
In addition to the main assumptions (f), ${\rm (D_{pn})}$-${\rm (D_{np})}$ and (g) stated in the Introduction, we also need for some results two further regularity conditions on the product of $Dg$ at the boundary of the interval $[0,1]$, which are stated using the above notation:
\begin{equation}
\label{Dg}
D^+Dg(0) < +\infty \ \mbox{ and } \ D^{-}Dg(1)<+\infty.
\end{equation}
In general, $\eqref{Dg}_1$ implies that problem $\eqref{e:zIntro}_1$ has a super-solution for sufficiently large $c$, see \cite{Malaguti-Marcelli_2003} in the case $f=0$, and hence it is solvable for those $c$. Condition $\eqref{Dg}_1$ is always satisfied if $D(0)=0$, while, if $D(0)\ne0$, it requires that $g$ is sublinear close to $0$; condition $\eqref{Dg}_2$ is commented analogously. At last, we denote the {\em difference quotient} of a function $F=F(\phi)$ with respect to a point $\phi_0$ as
\begin{equation}\label{e:DQ}
\delta(F,\phi_0)(\phi) := \frac{F(\phi)-F(\phi_0)}{\phi-\phi_0}.
\end{equation}

\smallskip

We first focus on the case {$(\rm D_{pn})$}. The construction of a wavefront to \eqref{e:E} takes place by properly joining two semi-wavefronts, each of them with an intrinsic threshold. The existence of such semi-wavefronts in the region $(0,\alpha)$ has been done in \cite[Theorem 2.2]{BCM1}; the main content of that result is the following. Under (f), {$(\rm D_{pn})$}, (g) and $\eqref{Dg}_1$ there exists a real number, denoted by $c^*_{p,r}$, satisfying
\begin{equation}
\label{e:c*pr}
\max\left\{
\sup_{(0,\alpha]} \delta(f,0),
h(0) + 2\sqrt{D_{+}Dg(0)}
 \right\}
 \leq
 \\
 c^*_{p,r}
  \leq
   \sup_{(0,\alpha]} \delta(f,0) + 2 \sqrt{\sup_{(0,\alpha]}\delta(Dg,0)},
\end{equation}
such that Equation \eqref{e:E} admits strict semi-wavefronts to $0$, connecting $\alpha$ to $0$, with speed $c$ if and only if $c\ge c^*_{p,r}$. It is worth noting that the left- and the right-hand side of \eqref{e:c*pr} describe a non-empty interval (possibly degenerating to a single point) of real numbers, as a direct inspection trivially shows. Moreover, in \cite[Theorem 3.1]{Marcelli-Papalini} the authors proved that in case $Dg$ differentiable at $\phi=0$ (e.g. in case $D(0)=0$) the second addend of the right-hand side of \eqref{e:c*pr} can be further enhanced by replacing $\delta(Dg,0)(\phi)$ with its mean value in $(0,\phi)$, that is
\begin{equation}
\label{e:MPintro}
c^*_{p,r} \le \sup_{(0,\alpha]} \delta(f,0) +2\sqrt{\sup_{\phi\in(0,\alpha]}\frac1{\phi}\int_{0}^{\phi} \frac{D(s)g(s)}{s}\,ds}.
\end{equation}
We warn the reader that those semi-wavefronts are proved to be {\em intrinsically} nonunique, i.e., nonuniqueness holds even understanding profiles differing by a horizontal shift as a same profile (see Proposition \ref{p:swf to zero}).
The subscripts $p,r$ in $c^*_{p,r}$ mean that we are considering a case where $D$ is {\em positive} in an interval $I_+$ and vanishes in the {\em right} extremum of $I_+$,
while $g$ vanishes at the opposite extremum. Similarly, we shall use the notation $c^*_{n,l}$, $c^*_{p,l}$ and $c^*_{n,r}$.

An analogous result for semi-wavefronts from $1$, connecting $1$ to $\alpha$, is first deduced in this paper. Assume {\rm (f)}, {$(\rm D_{pn})$}, {\rm (g)} and $\eqref{Dg}_2$. There exists $c^*_{n,l}\in \R$ satisfying
\begin{equation}
\label{e:c*nl}
\max\left\{\sup_{[\alpha,1)}\delta(f,1),
h(1) + 2\sqrt{D_{-}Dg(1)} \right\}
\leq
c^*_{n,l}
\leq
\sup_{[\alpha,1)}\delta(f,1) + 2 \sqrt{\sup_{[\alpha,1)}\delta(Dg,1)},
\end{equation}
such that Equation \eqref{e:E} admits strict semi-wavefronts from $1$, connecting $1$ to $\alpha$, with speed $c$, if and only if $c\geq c^*_{n,l}$; see Proposition \ref{prop:swf from one}. Also in this case profiles are intrinsically nonunique.

We now present our main results on {\em wavefronts}. Define, see Figure \ref{f:c1},
\begin{equation}
\label{e:c*pn}
c^*_{pn} :=\max\left\{c^*_{p,r}, c^*_{n,l}\right\}.
\end{equation}

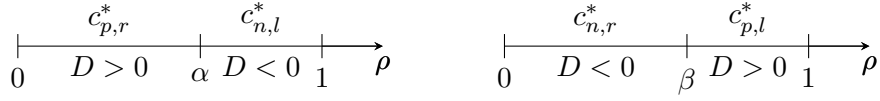
\begin{figure}[htb]
\begin{center}
\begin{tikzpicture}[>=stealth, scale=0.8]
\draw[->] (5,0) --  (6,0) node[below]{$\rho$} coordinate (x axis);
\draw (0,0) -- (3,0) node[above,midway]{$c_{p,r}^*$} node[below,midway]{$D>0$};
\draw (3,0) -- (5,0) node[above,midway]{$c_{n,l}^*$} node[below,midway]{$D<0$};
\draw[->] (5,0) --  (6,0) node[below]{$\rho$} coordinate (x axis);

\draw (0,-0.2) node[below]{$0$} -- (0,0.2);
\draw (5,-0.2) node[below]{$1$} -- (5,0.2);
\draw (3,-0.2) node[below]{$\alpha$}-- (3,0.2);

\begin{scope}[xshift=8cm]
\draw[->] (5,0) --  (6,0) node[below]{$\rho$} coordinate (x axis);
\draw (0,0) -- (3,0) node[above,midway]{$c_{n,r}^*$} node[below,midway]{$D<0$};
\draw (3,0) -- (5,0) node[above,midway]{$c_{p,l}^*$} node[below,midway]{$D>0$};
\draw[->] (5,0) --  (6,0) node[below]{$\rho$} coordinate (x axis);

\draw (0,-0.2) node[below]{$0$} -- (0,0.2);
\draw (5,-0.2) node[below]{$1$} -- (5,0.2);
\draw (3,-0.2) node[below]{$\beta$}-- (3,0.2);

\end{scope}
\end{tikzpicture}

\end{center}
\caption{\label{f:c1}{The thresholds $c^*_{p,r}$, $c^*_{n,l}$ used in \eqref{e:c*pn} and $c^*_{n,r}$, $c^*_{p,l}$, used below in \eqref{e:c*np}}.}
\end{figure}

For what concerns an estimate on $c^*_{pn}$, see Remark \ref{rem:c*pn}. We introduce the quantity $s_\pm (\gamma,c)$, defined formally by
\begin{equation}
\label{e:s}
s_\pm(\gamma,c):=\frac1{2}\left[ h(\gamma)-c \pm \sqrt{\left(h(\gamma)-c\right)^2-4 \dot{D}(\gamma)g(\gamma)} \right].
\end{equation}
In the next result, we make use of \eqref{e:s} with $\gamma=\alpha$. In this particular case, $s_-(\alpha, c)$ is well-defined since $\dot{D}(\alpha)g(\alpha)\le 0$.

\begin{theorem}
\label{thm:wf}
Assume {\rm (f)}, {\rm (g)}, {$\rm (D_{pn})$} and \eqref{Dg}. Equation \eqref{e:E} admits a (unique up to space shifts) wavefront, with speed $c$ and profile $\phi$ satisfying \eqref{e:infty}, if and only if $c\ge c^*_{pn}$.

For such $c$, we have $\phi\rq{}(\xi)<0$ when $\phi(\xi) \in (0,1)\setminus \{\alpha\}$; there exists a unique $\xi_\alpha \in \R$ such that $\phi(\xi_\alpha)=\alpha$ and
 \begin{equation}
 \label{e:phip(xialpha)}
\phi\rq{}(\xi_\alpha^\pm)=
\begin{cases}
\frac{g(\alpha)}{s_-(\alpha,c)}<0 \ &\mbox{ if } \ \dot{D}(\alpha)<0 \ \mbox{ or } \ c>h(\alpha),\\
-\infty \ &\mbox{ if } \ \dot{D}(\alpha)=0  \mbox{ and } c\leq h(\alpha).
\end{cases}
\end{equation}

At last, it holds that:

\begin{enumerate}[(i)]

\item if $D(0)D(1)<0$, then $\phi$ is strictly decreasing and hence classical in $\R\setminus\left\{\xi_\alpha\right\}$;

\item if $D(0)D(1)=0$ and $c>c^*_{pn}$, then $\phi$ is classical in $\R\setminus\{\xi_\alpha\}$; $\phi$ is strictly decreasing if
\[
c>\max\left\{h(0)+D^+g(0), \ h(1)+D^-\{-g\}(1)\right\};
\]
\item if $D(0)=0$ and $c=c_{pn}^*=c^*_{p,r}>h(0)$, then $\phi$ is sharp at $0$ (reached at some $\xi_0 > \xi_\alpha$) and if $D(1)=0$ and $c=c_{pn}^*=c^*_{n,l}>h(1)$, then $\phi$ is sharp at $1$ (reached at some $\xi_1 < \xi_\alpha$). In these cases
we have
\begin{eqnarray*}
 \phi\rq{}(\xi_0^-)=
\left\{
\begin{array}{ll}
\frac{h(0)-c^*_{p,r}}{\dot{D}(0)} <0 &\mbox{ if }\dot{D}(0)>0,\\[2mm]
-\infty \ &\mbox{ if } \dot{D}(0)=0,
\end{array}
\right. \  \phi\rq{}(\xi_1^+)=
\left\{
\begin{array}{ll}
\frac{h(1)-c^*_{n,l}}{\dot{D}(1)} <0 &\mbox{ if }\dot{D}(1)>0,\\[2mm]
-\infty &\mbox{ if }\dot{D}(1)=0.
\end{array}
\right.
\end{eqnarray*}
\end{enumerate}
\end{theorem}

We refer to Figure \ref{f:f2} for a pictorial representation of Theorem \ref{thm:wf}.
\begin{figure}[htb]
\begin{center}

\begin{tikzpicture}[>=stealth, scale=0.8]

\draw[->] (0,0) --  (10,0) node[below]{$\xi$} coordinate (x axis);
\draw[->] (4,0) -- (4,4) node[right]{$\phi$} coordinate (y axis);
\draw (0,3) -- (10,3);
\draw(4,3.3) node[left]{\footnotesize{$1$}};
\draw[dotted] (4,1.15) node[left=4,above=-2]{\footnotesize{$\alpha$}} -- (7,1.15);
\draw[dotted] (4,1.15) -- (3.2,1.15);
\draw[dotted] (5.2,0) node[below]{\footnotesize{$\xi_\alpha^1$}} -- (5.2,1.2);
\draw[dotted] (3.2,0) node[below]{\footnotesize{$\xi_\alpha^3$}} -- (3.2,1.2);
\draw[dotted] (7,0) node[below]{\footnotesize{$\xi_\alpha^2$}} -- (7,1.15);
\draw[dotted] (3,0) node[below=9, left=-3]{\footnotesize{$\xi_1^2$}} -- (3,3);
\draw[thick] (0,2.8) .. controls (3.5,2.8) and (5,0.2) .. node[above,midway]{$\phi^1$} (10,0.2); 
\draw[thick] (0,3) -- node[above, midway]{$\phi^2$} (3,3);
\draw[thick] (3,3) .. controls (3.8,2.5) and (7,2.5) .. (7,1.15); 
\draw[thick] (7,1.15) .. controls (7,0.4) and (9,0.3) .. node[above, near end]{$\phi^2$} (10,0.3); 
\draw[thick] (0,2.6) .. controls (2,2.6) and (3.8,1) .. node[below,near start]{$\phi^3$} (3.8,0) node[below]{\footnotesize{$\xi_0^3$}}; 
\draw[thick] (3.8,0) -- node[below, near end]{$\phi^3$}(10,0);
\end{tikzpicture}

\end{center}
\caption{\label{f:f2}{Some possible wavefronts joining $1$ with $0$ in case $\rm (D_{pn})$: a classical wavefront $\phi^1$ (for $c>c_{pn}^*$ and either $\dot D(\alpha)<0$ or $c>h(\alpha)$); a wavefront $\phi^2$ which is sharp at $1$, with finite right derivative at $\xi_1^2$ and $(\phi^2)'(\xi_\alpha^2)=-\infty$ (for $c=c_{n,l}^*>h(1)$, $D(1)=0$, $\dot D(1)>0$, $\dot D(\alpha)=0$ and $c\le h(\alpha)$); a wavefront $\phi^3$ which is sharp at $0$ with $(\phi^3)'(\xi_0^3)=-\infty$ (for $c=c_{p,r}^*> h(0)$, $D(0)=0=\dot D(0)$).}}
\end{figure}
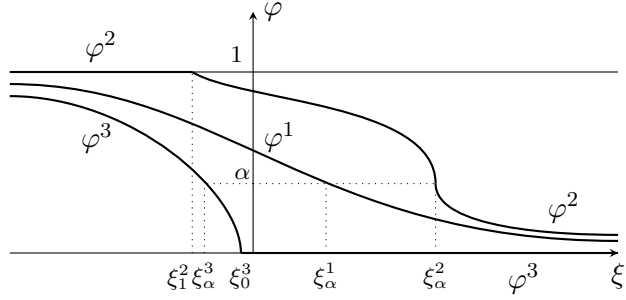
%
Theorem \ref{thm:wf} extends \cite[Theorem 1]{Maini-Malaguti-Marcelli-Matucci2006} to the case of a non-zero convection term $f$; if $f=0$ the estimates on $c_{pn}^*$ deduced from \eqref{e:c*pr} and \eqref{e:c*nl} coincide with those in \cite{Maini-Malaguti-Marcelli-Matucci2006}. Estimates \eqref{e:c*pr} and \eqref{e:c*nl} imply $c_{pn}^*\ge \max\left\{h(0), h(1)\right\}$. However, if $h(\phi)\ge h(0)$ for every $\phi$ in a right neighborhood of $0$, then $c_{p,r}^*>h(0)$, see \cite[Remark 6.4]{BCM1}, and so $c_{pn}^*>h(0)$. An analogous remark holds for the point $1$. This means that the lower estimate on $c_{pn}^*$ is not sharp, in general. Therefore we also extend the corresponding result of \cite{Maini-Malaguti-Marcelli-Matucci2006}, since they proved that the threshold $c^*$, which plays the role of $c_{pn}^*$ here, satisfies the stricter estimate $c^*>0$.

On the other hand, a new result that could not occur in \cite{Maini-Malaguti-Marcelli-Matucci2006} is that we can have $\phi'(\xi_\alpha)=-\infty$, see $\eqref{e:phip(xialpha)}_2$, while in the case $f=0$ profiles are always $C^1$ at $\xi_\alpha$. Explicit examples of fronts passing from a positive- to a negative-diffusivity region such that $\eqref{e:phip(xialpha)}_2$ occurs, are given in Example \ref{ex:alpha infinity slope}. From a formal point of view, if $f=0$ then this could occur if $c\le0$, see $\eqref{e:phip(xialpha)}_2$, while in that case $c>0$ holds; more rigorously, see \cite[(29)]{Maini-Malaguti-Marcelli-Matucci2006}. We emphasize that this phenomenon does not depend on the change of sign of $D$ at $\alpha$, but on the occurrence of the two conditions in $\eqref{e:phip(xialpha)}_2$, see \cite[Remark 9.3]{BCM1}. Moreover, just to get an insight on the problem, assume $\dot D(\alpha)=0$; in order that $\eqref{e:phip(xialpha)}_2$ takes place it is necessary that
\begin{equation}\label{e:batist}
\max\bigl\{\sup_{(0,\alpha]}\delta(f,0),\sup_{[\alpha,1)}\delta(f,1)\bigr\}\le h(\alpha),
\end{equation}
so that we have room for $c$. Condition \eqref{e:batist} has a simple geometric interpretation: the slope of the tangent to the graph of $f$ at $\alpha$ must be larger that the slope of any chord joining $(0,0)$ with any other point of the graph of $f$, in the interval in $(0,\alpha]$, and analogously for the interval $[\alpha,1)$. We easily see that $\sup_{(0,\alpha]}\delta(f,0)\le h(\alpha)$ fails if $f$ is concave in $[0,\alpha]$, while it holds if $f$ is convex; the converse result holds for the other condition. This means that $\eqref{e:phip(xialpha)}_2$ may hold only if $f$ changes its convexity. Moreover, if $f$ changes convexity only once, for instance at $\alpha$, then $f$ must be convex in $[0,\alpha]$ and concave in $[\alpha,1]$, and not conversely. In other words, the profile may become vertical  if, at least in some subintervals, the behavior of $f$ strongly contrast that of $D$.

We now comment on {\em (iii)}. The lack of results for the cases $c^*_{p,r}=h(0)$ or $c^*_{n,l}=h(1)$ is due to the fact that in these extremal cases the regularity depends on further properties of $D$ and $g$. Indeed, under the mild assumptions of Theorem \ref{thm:wf} the profile $\phi$ can be either sharp or classical when reaches the equilibria $0$ or $1$ (for explicit examples {\em at $0$}, we refer to \cite[Remark 10.1]{BCM1}; the discussion at $1$ is analogous).

We focus now on $\rm (D_{np})$. Condition {$\rm (D_{np})$} is specular to {$\rm (D_{pn})$}, in an obvious sense: if $D$ satisfies $\rm (D_{pn})$ then $-D$ satisfies $\rm (D_{np})$. Despite this fact, the results in this case only partially mimic those of Theorem \ref{thm:wf}. In particular, if we focus on the interval $(\beta,1)$, where $D>0$, the contrast with \cite[Theorem 2.2]{BCM1} emerges evident. Apart from a trivial horizontal translation, we are under the hypotheses considered in \cite[Theorem 2.7]{CM-DPDE} (compare  {$\rm (D_{np})$}-(g) in the interval $(\beta,1)$ with the corresponding assumptions of \cite[Theorem 2.7]{CM-DPDE}). In {\em contrast} with \cite[Theorem 2.2]{BCM1} cited before, \cite[Theorem 2.7]{CM-DPDE} affirms that, for {\em each} $c\in\R$, Equation \eqref{e:E} admits (unique up to space shifts) strict semi-wavefronts from $1$, connecting $1$ to $\beta$. Nevertheless, the system still admits a threshold speed, in the following sense. There exists $c^*_{p,l}\in \R$  such that the (unique up to space shifts) non-increasing wave profile $\phi$, defined maximally in $(-\infty,\xi_\beta)$, satisfies \cite[Theorem 2.6]{CM-DPDE}
 \begin{equation}
 \label{e:threshold cm-dpde condition}
 \left(D(\phi)\phi\rq{}\right)(\xi_\beta^-)=
 \left\{
\begin{array}{ll}
0 \ &\mbox{ if } \ c\ge c^*_{p,l},\\
\ell <0 \ &\mbox{ if }\ c<c^*_{p,l}.
\end{array}
\right.
 \end{equation}
By improving the estimates of $c^*_{p,l}$ in \cite{CM-DPDE}, as for \eqref{e:c*pr}-\eqref{e:MPintro}, it results that $c^*_{p,l}$ must satisfy
  \begin{multline}
 \label{e:c*pl}
 \max\left\{\sup_{(\beta,1]}\delta(f,\beta),h(\beta)+2\sqrt{\dot{D}(\beta)g(\beta)}\right\} \le c^*_{p,l} \le
 \\
  \sup_{(\beta,1]}\delta(f,\beta) +  2\sqrt{\sup_{\phi \in (\beta,1]}\frac{1}{\phi-\beta}\int_{\beta}^{\phi}\frac{D(s)g(s)}{s-\beta}\,ds}.
 \end{multline}
 See Proposition \ref{prop:cm-dpde}. Similarly, Equation \eqref{e:E} admits a (unique up to shifts) strict semi-wavefront to $0$, connecting $\beta$ to $0$, for every $c\in \R$. Moreover, there exists $c^*_{n,r} \in \R$ satisfying
  \begin{multline}
 \label{e:c*nr}
 \max\left\{\sup_{[0,\beta)}\delta(f,\beta),h(\beta)+2\sqrt{\dot{D}(\beta)g(\beta)}\right\} \le c^*_{n,r} \le
 \\
  \sup_{[0,\beta)}\delta(f,\beta) + 2\sqrt{\sup_{\phi \in [0,\beta)}\frac1{\beta-\phi}\int_{\phi}^{\beta}\frac{D(s)g(s)}{s-\beta}\,ds},
 \end{multline}
 such that the (unique up to shifts) non-increasing profile $\phi$, defined maximally in $(\xi_\beta, +\infty)$, satisfies
 \begin{equation}
 \label{e:threshold cm-dpde condition2}
 \left(D(\phi)\phi\rq{}\right)(\xi_\beta^+)=
 \left\{
\begin{array}{ll}
0 \ &\mbox{ if } \ c\ge c^*_{n,r},\\
s>0 \ &\mbox{ if }\ c< c^*_{n,r},
\end{array}
\right.
 \end{equation}
see Proposition \ref{prop:D1minus}. We set (see Figure \ref{f:c1})
\begin{equation}
\label{e:c*np}
c^*_{np}:=\max\left\{c^*_{n,r}, c^*_{p,l}\right\}.
\end{equation}

%
%

In the next theorem, $s_\pm(\beta, c)$ is given by \eqref{e:s}. Note that, in spite of $\dot{D}(\beta)\ge 0$, $s_\pm(\beta,c)$ is well-defined since in virtue of \eqref{e:c*np} and \eqref{e:c*pl} (or \eqref{e:c*nr}), we clearly have $(h(\beta)-c)^2 \ge (h(\beta) - c_{n,p}^*)^2 \ge 4\dot{D}(\beta)g(\beta)$.
\begin{theorem}
\label{thm:wf2}
Assume {\rm (f)}, {\rm (g)}, {$\rm (D_{np})$}. Then, Equation \eqref{e:E} admits a (unique up to space shifts) wavefront, with speed $c$ and profile $\phi$ satisfying \eqref{e:infty}, if and only if $c\ge c^*_{np}$.

For such $c$, we have $\phi\rq{}<0$ if $\phi \in (0,1)\setminus\left\{\beta\right\}$. There exists a unique $\xi_\beta \in\R$ such that $\phi(\xi_\beta)=\beta$ and
\begin{multline}
\label{e:reg at beta}
\phi\rq{}({\xi_\beta}^-)=
\begin{cases}
\frac{g\left(\beta\right)}{s_-(\beta,c)} \ &\mbox{ if } \ c>c^*_{p,l},\\[2mm]
\frac{g\left(\beta\right)}{s_+(\beta,c^*_{p,l})} \ &\mbox{ if } \ c=c^*_{p,l} \ \mbox{ and } \ \dot{D}(\beta)>0,\\
-\infty \ &\mbox{ if } \ c=c_{p,l}^* \ \mbox{ and } \ \dot{D}(\beta)=0,
\end{cases}
\\
\phi\rq{}({\xi_\beta}^+)=
\begin{cases}
\frac{g\left(\beta\right)}{s_-(\beta,c)} \ &\mbox{ if } \ c>c^*_{n,r},\\[2mm]
\frac{g\left(\beta\right)}{s_+(\beta,c^*_{n,r})} \ &\mbox{ if } \ c=c^*_{n,r} \  \mbox{ and } \ \dot{D}(\beta)>0,\\
-\infty \ &\mbox{ if } \ c=c_{n,r}^* \ \mbox{ and } \ \dot{D}(\beta)=0.
\end{cases}
\end{multline}

\noindent
At last, the following holds true:
\begin{enumerate}[(i)]

\item if $D(0)D(1)<0$ then $\phi$ is classical in $\R\setminus\left\{\xi_\beta\right\}$;

\item if $D(0)=0$, and either $c>h(0)$ or $c=h(0)$ and $\dot{D}(0)<0$, then $\phi$ is classical at $0$; if $D(1)=0$, and either $c>h(1)$ or $c=h(1)$ and $\dot{D}(1)<0$ then $\phi$ is classical at $1$;

\item if $D(0)=0$ and $c<h(0)$, then $\phi$ is sharp at $0$ (reached at some $\xi_0 > \xi_\beta$);
if $D(1)=0$ and $c<h(1)$, then $\phi$ is sharp at $1$ (reached at some $\xi_1 < \xi_\beta$). In these cases
we have
\begin{eqnarray*}
 \phi\rq{}(\xi_0^-)=
\left\{
\begin{array}{ll}
\frac{h(0)-c}{\dot{D}(0)}&\mbox{ if }\dot{D}(0)<0,\\[2mm]
-\infty \ &\mbox{ if } \dot{D}(0)=0,
\end{array}
\right. \  \phi\rq{}(\xi_1^+)=
\left\{
\begin{array}{ll}
\frac{h(1)-c}{\dot{D}(1)} &\mbox{ if }\dot{D}(1)<0,\\[2mm]
-\infty &\mbox{ if }\dot{D}(1)=0.
\end{array}
\right.
\end{eqnarray*}
\end{enumerate}
\end{theorem}

\begin{figure}[htb]
\begin{center}

\begin{tikzpicture}[>=stealth, scale=0.8]

\draw[->] (0,0) --  (10,0) node[below]{$\xi$} coordinate (x axis);
\draw[->] (4,0) -- (4,4) node[right]{$\phi$} coordinate (y axis);
\draw (0,3) -- (10,3);
\draw(4,3.3) node[left]{\footnotesize{$1$}};
\draw[dotted] (4,1.5) node[right=5,below=-2]{\footnotesize{$\beta$}} -- (10,1.5);
\draw[dotted] (0,1.5) -- (4,1.5);
\draw[thick] (0,2.5) .. controls (1.5,2.4) and (0,0.2) .. node[below,very near start]{$\phi^1$} (10,0.1); 
\draw[thick] (0,2.6) .. controls (3,2.3) and (3.5,2.5) .. node[below,midway]{$\phi^2$} (3.5,1.5); 
\draw[thick] (3.5,1.5) .. controls (3.5,0.9) and (4,0.2) .. (10,0.2); 
\draw[thick] (0,2.7) .. controls (3,2.5) and (3.5,2.5) .. node[above,very near end]{$\phi^3$} (5,1.5); 
\draw[thick] (5,1.5) .. controls (5,1) and (5,0.3) .. (10,0.3); 
\draw[thick] (0,2.8) .. controls (3,2.8) and (6,2.8) .. (6,1.5); 
\draw[thick] (6,1.5) .. controls (7,0.5) and (7,0.5) .. node[above, midway]{$\phi^4$} (10,0.4); 
\draw[thick] (0,2.9) .. controls (3,2.9) and (6,2.8) .. (7.5,1.5); 
\draw[thick] (7.5,1.5) .. controls (8,0.5) and (8,0.5) .. node[above,near end]{$\phi^5$} (10,0.5); 
\end{tikzpicture}

\end{center}
\caption{\label{f:f2s2}{Some possible wavefronts joining $1$ with $0$ in case $\rm (D_{np})$; Profiles are labelled according to the cases (1)--(4) of Remark \ref{rem:reg at beta}; $\phi^5$ occurs in both cases (3) and (4). For simplicity we only represented strictly monotone profiles.}}
\end{figure}
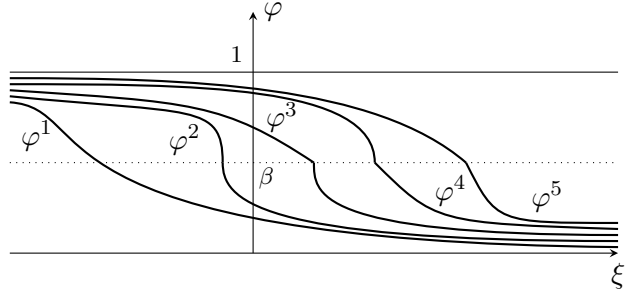

About the regularity of $\phi$ at $\xi_\beta$, see also Remark \ref{rem:reg at beta}. From \eqref{e:c*pl} and \eqref{e:c*nr}, we have $c^*_{np} \ge h(\beta)$ and $c^*_{np}>h(\beta)$ if $\dot{D}(\beta)>0$. Hence,  all quantities represented as $s_\pm(\beta,\cdot)$ in \eqref{e:reg at beta} are negative and finite real values.

Theorem \ref{thm:wf2} above extends \cite[Theorem 1]{Bao-Zhou2014} to the case of a non-zero convection term $f$; if $f=0$, the estimates on $c_{np}^*$ deduced from \eqref{e:c*pl} and \eqref{e:c*nr} improve those in \cite{Bao-Zhou2014}, not only because of a flaw  in the upper estimate in $c^*$ (the analog of $c_{np}^*$) in \cite{Bao-Zhou2014}, see formula (14) there, but also because the more precise estimates from \cite{BCM1} and \cite{Marcelli-Papalini} are involved here (and not in \cite{Bao-Zhou2014}). Moreover, by $c_{np}^*\ge h(\beta)$, with a reasoning as above Theorem \ref{thm:wf} we can further show, as in \cite{Bao-Zhou2014}, that $c^*_{np}>0$ when $f=0$. In Theorem \ref{thm:wf2} we also prove the existence of sharp profiles at either $0$ or $1$ even if $c>c_{np}^*$, see case {\em (iii)} above. Again, this is due to the  presence of $f$. Indeed, from a formal point of view, by the estimates in {\em (iii)} it follows that, if $f=0$, then to have sharp profiles one needs $c<0$, while $c$ is always positive in this case.

We point out that, if $c^*_{p,l} > c^*_{n,r}$ and $\phi$ is the profile corresponding to $c=c^*_{np}=c^*_{p,l}$ given by Theorem \ref{thm:wf2}, then $\phi\rq{}(\xi_\beta^-)\neq \phi\rq{}(\xi_\beta^+)$. The same conclusion holds if $c^*_{p,l} < c^*_{n,r}$ and $c=c^*_{np}=c^*_{n,r}$. Both alternatives $c^*_{p,l} =c^*_{n,r}$ and $c^*_{p,l}\neq c^*_{n,r}$ can indeed occur: explicit examples are shown in Example \ref{ex:n-p gap}.
This suggests that Theorems \ref{thm:wf} and \ref{thm:wf2} produce two separate families of solutions. Moreover, in Example \ref{ex:essentially different}, we show that the thresholds $c^*_{pn}$ of Theorem \ref{thm:wf} and $c^*_{np}$ of Theorem \ref{thm:wf2} are {\em essentially} different, in the sense that taking opposite diffusivities do {\em not} produce necessarily $c^*_{pn} = c^*_{np}$. Roughly speaking, this is due essentially because \eqref{e:c*pr}--\eqref{e:c*nl} and \eqref{e:c*pl}--\eqref{e:c*nr} are unrelated estimates.

\smallskip


Also notice the different role played by the sub-thresholds $c_{p,r}^*$, $c_{n,l}^*$ and $c_{p,l}^*$, $c_{n,r}^*$: the former two discriminate the {\em existence} of the semi-wavefronts, the latter two the {\em regularity}. Indeed, the estimates \eqref{e:c*pr}, \eqref{e:c*nl} concern the equilibrium points $0,1$ of $g$, while the estimates \eqref{e:c*pl}, \eqref{e:c*nr} only concern the point $\beta$ where $D$ vanishes. The values of $\phi'$ at $\xi_\alpha$ or $\xi_\beta$: $\phi'(\xi_\alpha)$ is uniquely determined (being possibly $-\infty$), while we can have $\phi'(\xi_\beta^-)\ne \phi'(\xi_\beta^+)$. Moreover, in the case $\rm (D_{pn})$, a profile may be sharp only if $c=c_{pn}^*$; in case $\rm (D_{np})$, a profile can be sharp also if $c>c_{np}^*$. As a consequence of these facts, items {\em (i)}--{\em (iii)} in the two theorems are similar but far from being symmetric.

\section{Wavefronts with positive to negative diffusivities}\label{sec:dpn}
\setcounter{equation}{0}
In this section we assume condition {$(\rm D_{pn})$}. First, we discuss the existence of semi-wavefronts to $0$ and from $1$ in Proposition \ref{p:swf to zero} and in Proposition \ref{prop:swf from one}, respectively. Then, by pasting their profiles,  we provide the proof of Theorem \ref{thm:wf}.

In the next proposition $s_-(\alpha,c)$ is given by \eqref{e:s}.

\begin{proposition}
\label{p:swf to zero}
Assume {\rm (f)}, {\rm (g)}, {$(\rm D_{pn})$},  $\eqref{Dg}_1$.  Then, there exists $c_{p,r}^*$ satisfying \eqref{e:c*pr} such that Equation \eqref{e:E} has strict semi-wavefronts to $0$, connecting $\alpha$ to $0$, if and only if $c\geq c^*_{p,r}$. If $\phi$ is the  non-increasing profile of one of them, then
\begin{equation}
\label{strict monotonicity}
\varphi\rq{}(\xi) < 0 \ \mbox{ for any } \ 0<\varphi(\xi) < \alpha.
\end{equation}
For $c>c^*_{p,r}$, there exists $\beta(c)<0$ such that every profile is uniquely determined (up to space shifts) by the value
\begin{equation}
\label{e:ell}
\left(D\left(\phi\right)\phi\rq{}\right)(\xi_\alpha^+)=:\ell \in [\beta(c), 0],
\end{equation}
where $\xi_\alpha \in \R$ is such that $(\xi_\alpha, +\infty)$ is the maximal-existence interval of $\phi$.
\par
If $\phi_\ell$ is the profile satisfying \eqref{e:ell}, then $\phi_\ell\rq{}(\xi_\alpha^+)=-\infty$ if $\ell\in[\beta(c),0)$, while
\begin{equation}
\label{e:phip alpha}
\phi_0\rq{}(\xi_\alpha^+)=
\begin{cases}
\frac{g(\alpha)}{s_-(\alpha, c)} \ &\mbox{ if } \ \dot{D}(\alpha)<0 \ \mbox{ or } \ c>h(\alpha) ,\\
-\infty \ &\mbox{ if } \ \dot{D}(\alpha)=0  \mbox{ and } c\leq h(\alpha).
\end{cases}
\end{equation}
\end{proposition}
\begin{proof}
We refer to Figure \ref{f:f1b}, solid lines. The existence of these semi-wavefronts, the estimate of $c^*_{p,r}$ and condition \eqref{strict monotonicity} are discussed in \cite[Theorem 2.2]{BCM1}. The proof of \cite[Theorem 2.2]{BCM1} also includes the equivalence between the presence of a semi-wavefront with speed $c$ and the solvability of the problem \eqref{e:zIntro}$_1$
 for the same $c$, with $z(\varphi)=D(\varphi)\varphi'$. Conditions \eqref{e:ell} and \eqref{e:phip alpha} then come, respectively,  from  \cite[Proposition 5.1]{BCM1} and   \cite[(9.2)]{BCM1} (see also \cite[Remark 9.2]{BCM1}) which are directly formulated for problem \eqref{e:zIntro}$_1$.
\end{proof}

\begin{figure}[htb]
\begin{center}
\begin{tikzpicture}[>=stealth, scale=0.6]
\draw[->] (0,0)  --  (6,0) node[below]{$\rho$} coordinate (x axis);
\draw[->] (0,0) -- (0,4) coordinate (y axis);
\draw[thick] (0,1) .. controls (1,3) and (2,3) .. (3.2,0) node[right=5, above=0]{\footnotesize{$\alpha$}} node[near end, below=3]{\footnotesize{$D$}};
\draw[thick, dashed] (3.2,0) .. controls (4,-2) and (4.5,-2) .. (5,-1);
\draw[dotted] (5,-1) -- (5,0) node[below=6,right=-3]{\footnotesize{$1$}};
\draw[thick] (0,0) node[below]{\footnotesize{$0$}} .. controls (1,2) and (1.6,3) .. (3.2,2.7) node[very near end, above]{\footnotesize{$g$}}; 
\draw[thick, dashed] (3.2,2.7) .. controls (3.8,2.55) and (4.3,2) .. (5,0); 
\draw[dotted] (3.2,0) -- (3.2,2.7);
\begin{scope}[xshift=8cm]
\draw[->] (0,0) --  (8,0) node[below]{$\xi$} coordinate (x axis);
\draw[->] (4,0) -- (4,4) node[right]{$\phi$} coordinate (y axis);
\draw (0,2) -- (8,2);
\draw (0,3.6) -- (8,3.6);
\draw[thick] (3,2) .. controls (4,1) and (6,0.2) .. (8,0.2); 
\draw[thick, dashed] (0,3.4) .. controls (1.5,3.4) and (2,3) .. (3,2); 
\draw[dotted] (3,0)  node[below]{$\xi_\alpha$} -- (3,2);
\draw(4,2.3) node[left]{\footnotesize{$\alpha$}};
\draw(3,4.1) node[left]{\footnotesize{$1$}};
\end{scope}
\end{tikzpicture}
\end{center}
\caption{\label{f:f1b}{$D$ satisfies $(\rm D_{pn})$: Left: the plots of $D$, $g$ in $[0,\alpha]$ (solid lines) and in $[\alpha,1]$ (dashed lines); right, a corresponding profile.}}
\end{figure}
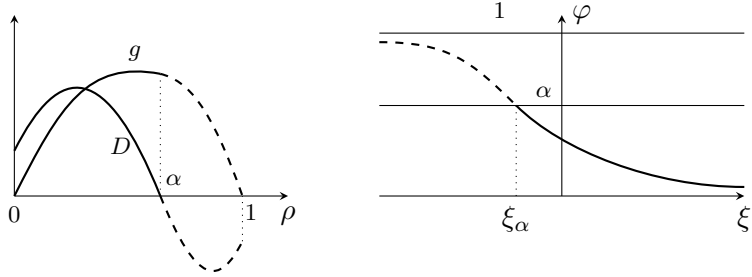

\begin{remark}
\label{rem:xi0}
{
\rm
Every semi-wavefront $\phi$ given in Proposition \ref{p:swf to zero} satisfies $\left(D\left(\phi\right) \phi\rq{}\right)(\xi_0^-)=0$,
(see e.g. \cite[formula (9.19)]{BCM1}), where $\xi_0= \sup\left\{ \xi>\xi_\alpha : \phi(\xi)>0 \right\} \in (\xi_\alpha, +\infty]$.
}
\end{remark}

\begin{proposition}
\label{prop:swf from one}
Assume {\rm (f)}, {\rm (g)}, {$(\rm D_{pn})$} and $\eqref{Dg}_2$. There exists $c^*_{n,l}\in \R$ satisfying \eqref{e:c*nl} such that Equation \eqref{e:E} admits strict semi-wavefronts from $1$, connecting $1$ to $\alpha$, with speed $c$, if and only if $c\geq c^*_{n,l}$. If $\phi$ is the non-increasing profile of one of such fronts, then it holds that
\begin{equation}
\label{strict monotonicity 2}
\phi\rq{}(\xi) < 0 \ \mbox{ if } \ \alpha < \phi(\xi) < 1.
\end{equation}
For $c>c^*_{n,l}$, there exists $\gamma(c)>0$ such that every profile is uniquely determined (up to space shifts) by the value
\begin{equation}
\label{e:ell2}
\left(D\left(\phi\right)\phi\rq{}\right)(\xi_\alpha^-)=:s \in [0,\gamma(c)],
\end{equation}
where $\xi_\alpha \in \R$ is such that $(-\infty, \xi_\alpha)$ is the maximal-existence interval of $\phi$.
\par
If $\phi_s$ is the profile satisfying \eqref{e:ell2}, then $\phi_s\rq{}(\xi_\alpha^-)=-\infty$ if $s\in (0, \gamma(c)]$, while
\begin{equation}
\label{e:222}
\phi_0\rq{}(\xi_\alpha^-)=
\begin{cases}
\frac{g(\alpha)}{s_-(\alpha, c)} \ &\mbox{ if } \ \dot{D}(\alpha)<0 \ \mbox{ or } \ c>h(\alpha) ,\\
-\infty \ &\mbox{ if } \ \dot{D}(\alpha)=0  \mbox{ and } c\leq h(\alpha).
\end{cases}
\end{equation}
\end{proposition}

\begin{proof}
We refer to Figure \ref{f:f1b}, dashed lines. Take, for $\phi \in [0,1]$,
\begin{equation}
\label{e:coefficients change}
\bar{D}\left(\phi\right):=-D\left(1-\phi\right), \ \bar{g}\left(\phi\right):=g\left(1-\phi\right), \ \bar{f}\left(\phi\right):=f(1)-f\left(1-\phi\right).
\end{equation}
Set $\bar{\alpha}:=1-\alpha$. Then $\bar{D}, \bar{g}, \bar{f}$ satisfy ${\rm (D_{pn})}$ with $\bar\alpha$ replacing $\alpha$, (g), (f) and $\eqref{Dg}_1$. Set $\bar{h}=\dot{\bar{f}}$. We apply Proposition \ref{p:swf to zero} to deduce that there exists a real value $c^*$ satisfying
\begin{equation*}
\max\left\{ \bar{h}(0)+2\sqrt{D_+\bar{D}\bar{g}(0)}, \sup_{(0,\bar\alpha]} \delta\left(\bar{f}, 0\right) \right\} \le
 c^* \le \sup_{(0,\bar\alpha]} \delta\left(\bar{f}, 0\right) + 2\sqrt{\sup_{(0,\bar\alpha]} \delta\left(Dg,0\right)}
\end{equation*}
such that there exist strict semi-wavefronts connecting $0$ to $\bar\alpha$, with speed $c$, if and only if $c\ge c^*$. In the formula above, we recall that $\delta \left(\cdot,\cdot\right)$ is defined in \eqref{e:DQ}. Direct manipulations show that the above chain of inequalities coincides with \eqref{e:c*nl}, in virtue of \eqref{e:coefficients change}. Fix $c\ge c^*$ and let $\phi_{\bar\alpha,0}$ be the profile of a semi-wavefront connecting $\bar\alpha$ to $0$, having speed $c$, given by Proposition \ref{p:swf to zero}. Moreover, assume that $\phi_{\bar\alpha,0}$ is maximally defined in $(\xi_{\bar\alpha},+\infty)$. Define then $\phi_{1,\alpha}:(-\infty,-\xi_{\bar\alpha})\to \R$ by
\begin{equation}
\label{e:swf change}
\phi_{1,\alpha}(\xi):=1-\phi_{\bar\alpha, 0} \left(- \xi\right) \ \mbox{ for } \ \xi <\xi_\alpha:=-\xi_{\bar\alpha}.
\end{equation}
Notice that, because $\phi_{\bar\alpha,0}$ connects $\bar\alpha$ to $0$, then \eqref{e:swf change} implies
\[
\lim_{\xi \to -\infty}\phi_{1,\alpha}(\xi)=1 \ \mbox{ and } \ \lim_{\xi\to \xi_{\alpha}^-}\phi_{1,\alpha}(\xi)=\alpha,
\]
thus $\phi_{1,\alpha}$ connects $1$ to $\alpha$. Also, from \eqref{e:swf change} and $\bar{D}\left(\phi_{\bar\alpha,0}\right)\phi_{\bar\alpha,0}\rq{} \in L^1\left(\xi_{\bar\alpha},\infty\right)$, we have $D\left(\phi_{1,\alpha}\right)\phi_{1,\alpha}\rq{}\in L^1\left(-\infty,\xi_{\alpha}\right)$. Since $\phi_{\bar\alpha,0}$ satisfies \eqref{e:ODE}, with $\bar{D}, \bar{g}$, $\bar{h}$, in $(\xi_{\bar\alpha},+\infty)$ for $c\ge c^*$, then direct computations show that $\phi_{1,\alpha}$ satisfies \eqref{e:ODE} in $(-\infty, \xi_{\alpha})$ with the same $c$. Solutions are always meant in the sense of Definition \ref{d:tws}.
\par
Viceversa, let $\phi$ be a profile of a strict semi-wavefront of \eqref{e:E}, connecting $1$ to $\alpha$ associated to some $c\in \R$, defined maximally in $(-\infty, 0)$. By setting
\begin{equation}
\label{e:swf change 2}
\phi_{\bar\alpha,0}(\tau):=1-\phi\left(- \tau\right) \ \mbox{ for } \ \tau > 0,
\end{equation}
we obtain the profile of a strict semi-wavefront, from $\bar\alpha$ to $0$, of Equation \eqref{e:ODE} (with $\bar{D}$, $\bar{g}$ and $\bar{f}$ as in \eqref{e:coefficients change}) associated to the speed $c$. By applying Proposition \ref{p:swf to zero} we deduce that $c\ge c^*$. The first part of the statement is hence proved, with $c^*_{n,l}=c^*$ defined above.

Finally, \eqref{strict monotonicity 2}--\eqref{e:222} follow by \eqref{e:swf change} and Proposition \ref{p:swf to zero}, for $\gamma(c):=-\beta(c)$.
\end{proof}

\begin{remark}
\label{rem:xi1}
{
\rm
 According to \eqref{e:swf change} and Remark \ref{rem:xi0}, every semi-wavefront $\phi$ given in Proposition \ref{prop:swf from one} satisfies $\left(D\left(\phi\right)\phi\rq{}\right)\left(\xi_1^+\right)=0$, where $\xi_1:=\inf\left\{ \xi< \xi_\alpha: \phi(\xi)<1\right\} \in [-\infty, \xi_\alpha)$.
}
\end{remark}

\begin{proofof}{Theorem \ref{thm:wf}}
Let $c^*_{p,r}$ and $c^*_{n,l}$ be the thresholds given in Proposition \ref{p:swf to zero} and \ref{prop:swf from one}, respectively, and define $c^*_{pn}$ as in \eqref{e:c*pn}.

First, we take $c\ge c^*_{pn}$ and prove that there is a wavefront to equation \eqref{e:E} that satisfies \eqref{e:infty} and has speed $c$. Associated to such a value of $c$, from Proposition \ref{p:swf to zero}, there exists a strict semi-wavefront $\phi_0$ (according to the notation introduced in the statement of Proposition \ref{p:swf to zero}) to $0$ with wave speed $c$, connecting $\alpha$ to $0$, such that
\begin{equation}
\label{e:xialpha phi}
\left(D\left(\phi_0\right)\phi_0\rq{}\right)(\xi_\alpha^+)=0,
\end{equation}
see \eqref{e:ell}. In \eqref{e:xialpha phi}, the value $\xi_\alpha$, which is finite because we are considering a {\em strict} semi-wavefront, is such that $\phi_0$ is maximally defined in $(\xi_\alpha, +\infty)$. Analogously, from Proposition \ref{prop:swf from one}, we have that there exists a semi-wavefront from $1$ with wave speed $c$, which connects $1$ to $\alpha$ and such that its profile $\psi_0$ realizes
\begin{equation}
\label{e:xialpha psi}
\left(D\left(\psi_0\right)\psi_0\rq{}\right)(\xi_\alpha^-)=0,
\end{equation}
see \eqref{e:ell2}. Here, we assumed that $\psi_0$ is maximally defined in $(-\infty,\xi_\alpha)$; this is always possible unless of shifting $\psi_0$. We refer to Remark \ref{rem:swf choice} for the reasons of the choices of $\phi_0$ and $\psi_0$.

Let $\xi_1, \xi_0 \in \R$ be such that
\begin{eqnarray*}
&\xi_0:= \sup\left\{\xi>\xi_\alpha: \phi_0(\xi)>0\right\}\in (\xi_\alpha, +\infty],\\
&\xi_1:=\inf\left\{\xi < \xi_\alpha: \psi_0(\xi) < 1\right\} \in [-\infty, \xi_\alpha),
\end{eqnarray*}
see Figure \ref{f:fphi12}. Define then the function  $\phi=\phi(\xi)$ by
\begin{equation}
\label{e:phi global}
\phi(\xi)=
\begin{cases}
\phi_0(\xi), \ \xi \ge \xi_\alpha,\\
\psi_0(\xi), \  \xi < \xi_\alpha.
\end{cases}
\end{equation}
It is clear that $\phi$ is well-defined and continuous in $(-\infty, +\infty)$; moreover, $\phi$ is non-increasing and  connects $1$ to $0$. Since both $\phi_0$ and $\psi_0$ satisfy \eqref{e:ODE} in their domains, then $\phi$ satisfies \eqref{e:ODE} pointwise in $(-\infty, \xi_1)\cup (\xi_1,\xi_\alpha)\cup (\xi_\alpha, \xi_0)\cup(\xi_0, +\infty)$.

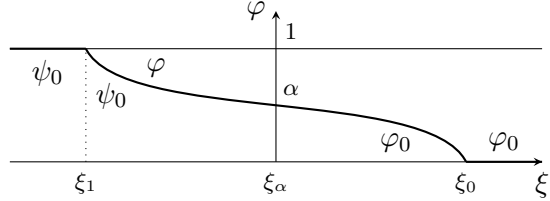
\begin{figure}[htb]
\begin{center}

\begin{tikzpicture}[>=stealth, scale=0.5]

\draw[->] (0,0) --  (7,0) node[below]{$\xi$} coordinate (x axis);
\draw (0,0) --  (-7,0);
\draw (-7,3) --  (7,3);
\draw[->] (0,0) node[below]{\footnotesize{$\xi_\alpha$}}-- (0,4) node[left]{$\phi$} coordinate (y axis);
%
\draw[thick] (-5,3) .. controls (-4,1) and (4,2)  .. node[very near start, below]{$\psi_0$} node[near end, below]{$\phi_0$} node[near start, above]{$\phi$} (5,0) node[below]{\footnotesize{$\xi_0$}};
\draw[thick] (-7,3) -- node[midway,below]{$\psi_0$} (-5,3);
%
\draw(0,1.5) node[right=0.2cm,above=-0cm]{\footnotesize{$\alpha$}};
\draw(0,3) node[right=0.2cm,above=-0.cm]{\footnotesize{$1$}};
\draw[dotted] (-5,0) node[below]{\footnotesize{$\xi_1$}}--  (-5,3);
\draw[thick] (5,0) -- node[midway, above]{$\phi_0$} (7,0);
\end{tikzpicture}

\end{center}
\caption{\label{f:fphi12}{Construction of the profile $\phi$ in the case $\xi_0,\xi_1\in\R$.}}
\end{figure}

Formulas \eqref{e:xialpha phi}, \eqref{e:xialpha psi} imply that $D(\phi)\phi\rq{}$ is continuously extended to $\xi=\xi_\alpha$; we have
\begin{equation}
\label{e:xi1 xi0}
\left(D\left(\psi_0\right)\psi_0\rq{}\right)(\xi_1^+)=0 \ \mbox{ and } \ \left(D\left(\phi_0\right)\phi_0\rq{}\right)(\xi_0^-)=0.
\end{equation}
Formula $\eqref{e:xi1 xi0}_2$ follows from Remark \ref{rem:xi0} applied to $\phi_0$ and $\eqref{e:xi1 xi0}_1$ follows from Remark \ref{rem:xi1} applied to $\psi_0$. Hence, we showed that $D(\phi)\phi\rq{}\in L^1_{\rm loc}(-\infty,+\infty)$. It remains to show that $\phi$ is a solution of \eqref{e:ODE} in $(-\infty, +\infty)$, according to Definition \ref{d:tws}. To this purpose, take $\zeta \in C_0^\infty(-\infty,\infty)$. Since $\phi$ is a distributional solution of \eqref{e:ODE} in both $(-\infty, \xi_\alpha)$ and $(\xi_\alpha, + \infty)$, then we can reduce to test \eqref{e:def-tw} when $\supp \zeta \subseteq [\xi_\alpha - \delta, \xi_\alpha+\delta] \subset (\xi_1, \xi_0)$, for some $\delta>0$. Hence, we need to test whether
 \begin{equation}
 \label{e:integral0}
 \int_{\xi_\alpha-\delta}^{\xi_\alpha+\delta} \left(D\left(\phi\right)\phi\rq{}-f\left(\phi\right)+c\phi\right)\zeta\rq{} -g\left(\phi\right)\zeta \, d\xi=0.
 \end{equation}
 Clearly, the left-hand side of \eqref{e:integral0} equals
 \begin{multline}
 \label{e:integral01}
 \int_{\xi_\alpha-\delta}^{\xi_\alpha}\left(D\left(\psi_0\right)\psi_0\rq{}-f\left(\psi_0\right)+c\psi_0\right)\zeta\rq{} -g\left(\psi_0\right)\zeta \, d\xi + \\
 \int_{\xi_\alpha}^{\xi_\alpha+\delta}\left(D\left(\phi_0\right)\phi_0\rq{}-f\left(\phi_0\right)+c\phi_0\right)\zeta\rq{} -g\left(\phi_0\right)\zeta \, d\xi.
 \end{multline}
 Let us focus on the former addend of \eqref{e:integral01}. We have:
 \begin{multline}
 \label{e:integral011}
  \int_{\xi_\alpha-\delta}^{\xi_\alpha}\left(D\left(\psi_0\right)\psi_0\rq{}-f\left(\psi_0\right)+c\psi_0\right)\zeta\rq{} -g\left(\psi_0\right)\zeta \, d\xi =\\
   \lim_{\delta>\eps \to 0^+} \int_{\xi_\alpha-\delta}^{\xi_\alpha- \eps} \left(D\left(\psi_0\right)\psi_0\rq{}-f\left(\psi_0\right)+c\psi_0\right)\zeta\rq{} -g\left(\psi_0\right)\zeta \, d\xi=\\
   \lim_{\delta>\eps\to 0^+}
    \left(\left(D\left(\psi_0\right)\psi_0\rq{}\right)\left(\xi_\alpha-\eps\right) - f\left(\psi_0\left(\xi_\alpha-\eps\right)\right)-c\psi_0\left(\xi_\alpha-\eps\right) \right)\zeta\left(\xi_\alpha-\eps\right)= \\
     \left(D\left(\psi_0\right)\psi_0\rq{}\right)\left(\xi_\alpha^-\right)\zeta\left(\xi_\alpha\right) - \left(f\left(\alpha\right)-c\alpha \right)\zeta\left(\xi_\alpha\right)=- \left(f\left(\alpha\right)-c\alpha \right)\zeta\left(\xi_\alpha\right).
 \end{multline}
Similar arguments involving the latter addend of \eqref{e:integral01} lead to
\begin{multline}
\label{e:integral021}
 \int_{\xi_\alpha}^{\xi_\alpha+\delta}\left(D\left(\phi_0\right)\phi_0\rq{}-f\left(\phi_0\right)+c\phi_0\right)\zeta\rq{} -g\left(\phi_0\right)\zeta \, d\xi = \\
-\left(D\left(\phi_0\right)\phi_0\rq{}\right)\left(\xi_\alpha^+\right)\zeta\left(\xi_\alpha\right) + \left(f(\alpha)-c\alpha\right) \zeta\left(\alpha\right)= \left(f(\alpha)-c\alpha\right) \zeta\left(\xi_\alpha\right).
\end{multline}
Putting together \eqref{e:integral011} and \eqref{e:integral021} proves \eqref{e:integral0}.

\smallskip
Conversely, we prove that if there exists a wavefront with speed $c$, whose profile $\phi$ is non-increasing and satisfies \eqref{e:ODE}-\eqref{e:infty}, then $c \ge c^*_{pn}$. Let $\xi_\alpha\in\R$ be such that $\phi(\xi_\alpha)=\alpha$. Such a $\xi_\alpha$ obviously exists since $\phi$ is continuous and satisfies \eqref{e:infty}. Furthermore, we have
\begin{equation}
\label{e:item ii}
\left\{\phi = \alpha\right\}=\left\{\xi_\alpha\right\}.
\end{equation}
Indeed, otherwise there exists an open set $J\subset \left\{\phi = \alpha\right\}$ where $\phi$ is constantly equal to $\alpha$. Thus, in $J$, equation \eqref{e:ODE} for $\phi$ reduces to $g(\alpha)=0$, which is clearly forbidden by (g). Then, we proved \eqref{e:item ii}.
\par
The function $\phi_{\alpha,0}:(\xi_\alpha, +\infty)\to (0,\alpha)$, defined by $\phi_{\alpha,0}=\phi$, is a strict semi-wavefront to $0$ with speed $c$, connecting $\alpha$ to $0$, and the function $\phi_{1,\alpha}:(-\infty, \xi_\alpha)\to (\alpha,1)$, defined by $\phi_{1,\alpha}=\phi$, is a strict semi-wavefront from $1$, with speed $c$, connecting $1$ to $\alpha$. From Propositions \ref{p:swf to zero} and \ref{prop:swf from one}, both $c\ge c^*_{p,r}$ and $c\ge c^*_{n,l}$ must occur. Hence, $c \ge c^*_{pn}$, and the first part of Theorem \ref{thm:wf} is proved.

\smallskip
To conclude the proof, we apply \cite[Corollary 9.4]{BCM1} to $\phi_{\alpha,0}$ (defined just above) and to $\phi_{\bar\alpha,0}$ (defined by \eqref{e:swf change 2}).
\end{proofof}

\begin{remark}[Estimates for $c^*_{pn}$]
\label{rem:c*pn}
{
\rm
We now explicitly provide estimates for the threshold $c^*_{pn}$ in \eqref{e:c*pn}. Obviously, $c^*_{pn}$ inherits the bounds, from above and below, for both $c^*_{p,r}$ and $c^*_{n,l}$. Such estimates are contained in Propositions \ref{p:swf to zero} and \ref{prop:swf from one}. We hence have:
\begin{align*}
c^*_{pn} &\ge \max\left\{ \sup_{(0,\alpha]} \delta(f,0), \sup_{[\alpha,1)}\delta(f,1), h(0)+2\sqrt{D_{+}Dg(0)}, h(1)+2\sqrt{D_{-}Dg(1)}\right\},
\\
c^*_{pn} &\le \max\left\{\sup_{(0,\alpha]}\delta(f,0) + 2 \sqrt{\sup_{(0,\alpha]}\delta(Dg,0)}, \sup_{[\alpha,1)}\delta(f,1) + 2\sqrt{\sup_{[\alpha,1)}\delta(Dg,1)}\right\}.
\end{align*}
If $f$ is identically zero, such estimates were given in \cite[formula (14)]{Maini-Malaguti-Marcelli-Matucci2006}.
}
\end{remark}

\begin{remark}
\label{rem:swf choice}
{
\rm
It is worth emphasizing that, given $c$ large enough, among the profiles in the families
\[
\{\phi_\ell:\ell \in [\beta(c),0]\}\quad \hbox{ and } \quad \{\psi_s: s \in[0,\gamma(c)]\},
\]
given by Propositions \ref{p:swf to zero} and \ref{prop:swf from one} respectively, we can only benefit from $\phi_0$ and $\psi_0$ to construct a wavefront as in Theorem \ref{thm:wf}. In particular, we can take advantage only of those profiles whose associated functions $z$ (solving \eqref{e:zIntro}) vanish at $\alpha^\pm$. In all the other cases, the pasting of a profile $\phi_\ell$ with a profile $\psi_s$ does not provide a solution (according to the distributional sense of Definition \ref{d:tws}) in a neighborhood of the matching point. Indeed, under these assumptions, \eqref{e:integral011} and \eqref{e:integral021} read respectively as
\begin{align}
\label{e:integral3_2}
\int_{\xi_\alpha-\delta}^{\xi_\alpha}\left(D\left(\psi_s\right)\psi_s\rq{}-f(\psi_s)+c\psi_s\right)\zeta\rq{}-g(\psi_s)\zeta\,d\xi & =
 \left[ s  + c\alpha -f(\alpha)\right]\zeta(\xi_\alpha),
\\
\label{e:integral4_2}
\int_{\xi_\alpha}^{\xi_\alpha+\delta}\left(D\left(\phi_\ell\right)\phi_\ell\rq{}-f(\phi_\ell)+c\phi_\ell\right)\zeta\rq{}-g(\phi_\ell)\zeta\,d\xi  &= \left[-\ell -c\alpha + f(\alpha)\right]\zeta(\xi_\alpha).
\end{align}
Thus, in place of \eqref{e:integral0}, putting together \eqref{e:integral3_2} and \eqref{e:integral4_2} gives
\begin{equation}
\label{e:integral5_2}
\int_{\xi_\alpha-\delta}^{\xi_\alpha+\delta} \left(D\left(\phi\right)\phi\rq{}-f(\phi)+c\phi\right)\zeta\rq{}-g(\phi)\zeta\,d\xi = \left(s-\ell\right) \zeta(\xi_\alpha),
\end{equation}
which vanishes for each arbitrary test function $\zeta$ if and only if $s=\ell=0$.

In the first instance, this seems to be suggested by the fact that if, formally,
\[
z(\alpha) = D\left(\phi(\xi_\alpha)\right)\phi'(\xi_\alpha^-) = D(\alpha)\phi'(\xi_\alpha^-)<0,
\]
then necessarily $\phi\rq{}(\xi_\alpha^-)=-\infty$, because $D(\alpha)=0$; the same remark holds if $w(\alpha)>0$. Nonetheless, the failure of the pasting is not due to a possibly infinite derivative of the profile at the matching point; indeed, Theorem \ref{thm:wf} {\em (ii)} shows that the wavefront $\phi$ {\em can} well have infinite slope at $\xi_\alpha$, see profile $\phi^2$ in Figure \ref{f:f2}.
}
\end{remark}

\section{Wavefronts with negative to positive diffusivities}\label{sec:dnp}
\setcounter{equation}{0}
 In this section we assume condition {$(\rm D_{np})$} and prove Theorem \ref{thm:wf2}.

The existence and regularity of semi-wavefronts from $1$, connecting $1$ to $\beta$, was obtained essentially in \cite[Theorem 2.7]{CM-DPDE} in the case $D(1)>0$ and in \cite[Theorems 2.3 and 2.5]{CdRM} when $D(1)=0$.
 We collect these results in the next proposition. The sharper estimate \eqref{e:c*pl}  for the threshold $c^*_{p,l}$
 comes from \cite[Corollary 5.4 and Remark 5.5]{BCM1}. The symbols $s_{\pm}(\alpha,c)$ are given by \eqref{e:s}. We refer to Figure \ref{f:f1bb}, solid lines.

\begin{proposition}
\label{prop:cm-dpde}
Assume {\rm (f)}, {\rm (g)}, {$\rm (D_{np})$}. Then, for every $c\in \R$, Equation \eqref{e:E} has a (unique up to shifts) strict semi-wavefront solution from $1$, connecting $1$ to $\beta$, with speed $c$ and profile $\phi$ defined in its maximal-existence interval $(-\infty, \xi_\beta)$, for some $\xi_\beta\in \R$. It holds that $\phi\rq{} <0$ if $\beta<\phi <1$.
There exists $c^*_{p,l}$ satisfying \eqref{e:c*pl} such that \eqref{e:threshold cm-dpde condition} holds true and we have
\begin{equation}
\label{e:phiprime betaminus}
\phi\rq{}(\xi_\beta^-) =
\begin{cases}
\frac{g(\beta)}{s_-\left(\beta, c\right)} \ &\mbox{ if } \ c> c^*_{p,l},\\[2mm]
\frac{g(\beta)}{s_+(\beta,c^*_{p,l})} \ &\mbox{ if }\ c= c^*_{p,l} \ \mbox{ and } \ \dot{D}(\beta)>0, \\
-\infty \ &\mbox{ if }\ c=c^*_{p,l} \ \mbox{ and } \ \dot{D}(\beta)=0 \ \mbox{ or } \ c<c^*_{p,l}.
\end{cases}
\end{equation}

Moreover, the following results hold.
\begin{enumerate}[(i)]
\item If $D(1)>0$, then $\phi$ is classical.

\item If $D(1)=0$ and either $c>h(1)$ or $c=h(1)$ and $\dot{D}(1)<0$, then $\phi$ is classical.

 \item If $D(1)=0$ and $c<h(1)$, then $\phi$ is sharp at $1$ (reached at some $\xi_1<\xi_\beta$) with
\begin{equation*}
 \phi\rq{}(\xi_1^+)=
\left\{
\begin{array}{ll}
\frac{h(1)-c}{\dot{D}(1)} <0 \ &\mbox{ if } \ \dot{D}(1)<0,\\[2mm]
-\infty \ &\mbox{ if } \ \dot{D}(1)=0.
\end{array}
\right.
\end{equation*}
\end{enumerate}
\end{proposition}

\begin{figure}[htb]
\begin{center}

\begin{tikzpicture}[>=stealth, scale=0.6]
\draw[->] (0,0)  --  (6,0) node[below]{$\rho$} coordinate (x axis);
\draw[->] (0,0) -- (0,4) coordinate (y axis);
\draw (0,0) -- (0,-1.2);
\draw[thick,dashed] (0,-1) .. controls (1,-3) and (2,-3) .. (3.2,0) node[right=5, below=0]{\footnotesize{$\beta$}};
\draw[thick] (3.2,0) .. controls (4,2) and (4.5,2) .. (5,1) node[very near start, right]{\footnotesize{$D$}};
\draw[dotted] (5,1) -- (5,0) node[below]{\footnotesize{$1$}};
\draw[thick,dashed] (0,0) .. controls (1,2) and (1.6,3) .. (3.2,2.7); 
\draw[thick] (3.2,2.7) .. controls (3.8,2.55) and (4.3,2) .. (5,0)  node[near start, above]{\footnotesize{$g$}}; 
\draw[dotted] (3.2,0) -- (3.2,2.7);

\begin{scope}[xshift=8cm]
\draw[->] (0,0) --  (8,0) node[below]{$\xi$} coordinate (x axis);
\draw[->] (4,0) -- (4,4) node[right]{$\phi$} coordinate (y axis);
\draw (0,2) -- (8,2);
\draw (0,3.6) -- (8,3.6);
\draw[thick] (0,3.4) .. controls (2,3.4) and (4,3) .. (5,2) ;
\draw[thick, dashed] (5,2) .. controls (6,1.4) and (6,0.2) .. (8,0.2) ;
\draw[dotted] (5,0)  node[below]{$\xi_\beta$} -- (5,2);
\draw(4,2.3) node[left=5,below=3]{\footnotesize{$\beta$}};
\draw(3,4.1) node[left]{\footnotesize{$1$}};
\end{scope}
\end{tikzpicture}
\end{center}
\caption{\label{f:f1bb}{$D$ satisfies $(\rm D_{np})$. Left: the plots of $D$, $g$ in $[0,\beta]$ (dashed lines) and $[\beta,1]$ (solid lines); right, a corresponding profile.}}
\end{figure}
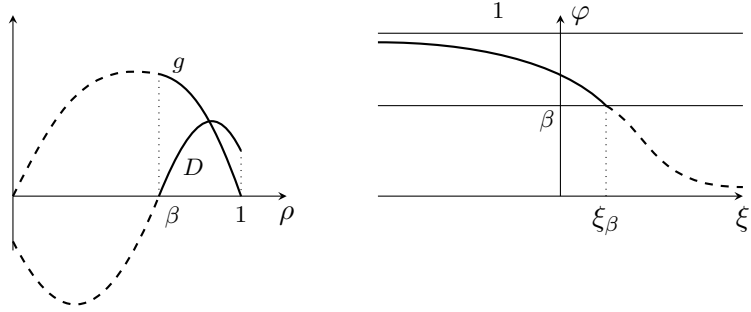
We discuss the existence of semi-wavefronts to $0$ in the next proposition.
\begin{proposition}
\label{prop:D1minus}
Assume {\rm (f)}, {\rm (g)}, {$\rm (D_{np})$}. Then, for every $c\in \R$ Equation \eqref{e:E} has a (unique up to space shifts) strict semi-wavefront solution to $0$, connecting $\beta$ to $0$, with speed $c$ and profile $\phi$ defined in its maximal-existence interval $(\xi_\beta, +\infty)$, for some $\xi_\beta \in \R$. It holds that $\phi\rq{} <0$ if $0<\phi <\beta$.
In addition, there exists $c^*_{n,r}\in\R$ satisfying \eqref{e:threshold cm-dpde condition2} and we have
\begin{equation}
\label{e:phiprime betaplus}
\phi\rq{}(\xi_\beta^+) =
\begin{cases}
\frac{g(\beta)}{s_-(\beta, c)} \ &\mbox{ if } \ c> c^*_{n,r},\\[2mm]
\frac{g(\beta)}{s_+(\beta, c_{n,r}^*)} \ &\mbox{ if }\ c=c^*_{n,r} \ \mbox{ and } \ \dot{D}(\beta)>0,\\
-\infty \ &\mbox{ if }\ c<c^*_{n,r} \ \mbox{ or } \ c=c^*_{n,r} \ \mbox{ and } \ \dot{D}(\beta)=0.
\end{cases}
\end{equation}

Moreover, the following results hold.
\begin{enumerate}[(i)]
\item If $D(0)<0$, then $\phi$ is classical.

\item If $D(0)=0$ and either $c>h(0)$ or $c=h(0)$ and $\dot{D}(0)<0$ then $\phi$ is classical.

 \item If $D(0)=0$ and $c<h(0)$ then $\phi$ is sharp at $0$ (reached at some $\xi_0>\xi_\beta$) with
\begin{equation*}
 \phi\rq{}(\xi_0^-)=
\left\{
\begin{array}{ll}
\frac{h(0)-c}{\dot{D}(0)} <0 \ &\mbox{ if } \ \dot{D}(0)<0,\\[2mm]
-\infty \ &\mbox{ if } \ \dot{D}(0)=0.
\end{array}
\right.
\end{equation*}
\end{enumerate}
\end{proposition}

\begin{proof}
We refer to Figure \ref{f:f1bb}, dashed lines. Let $\bar{D}$, $\bar{g}$ and $\bar{f}$ be defined by \eqref{e:coefficients change}. We already observed in the proof of Proposition \ref{prop:swf from one} that $\bar{g}$ and $\bar{f}$ still satisfy (g) and (f). In this case, instead, $\bar{D}$ satisfies ${\rm (D_{np})}$ with $\bar\beta:=1-\beta$ replacing $\beta$. Hence, Proposition \ref{prop:cm-dpde} applied to $\bar{D}$, $\bar{g}$ and $\bar{f}$ informs us that strict semi-wavefronts with speed $c\in\R$ connecting $1$ to $\bar\beta$ exist for every $c$. Let $\phi_{1,\bar\beta}$ be a profile of one of such fronts, defined in $(-\infty, \xi_{\bar\beta})$. Set $\xi_\beta:=-\xi_{\bar\beta}$ and
\begin{equation}
\label{e:change 2}
\phi_{\beta,0}(\xi):= 1- \phi_{1,\bar\beta}(-\xi) \ \mbox{ for } \ \xi >\xi_{\beta}.
\end{equation}
With arguments analogous to those in the proof of Proposition \ref{prop:swf from one}, it turns out that $\phi_{\beta,0}$ is the profile of a desired strict semi-wavefront of \eqref{e:E}, connecting $\beta$ to $0$. Hence, the first part of the statement is proved. By Proposition \ref{prop:cm-dpde} applied to $\bar{D}$, $\bar{g}$ and $\bar{f}$, in the interval $[\bar\beta,1]$, we obtain that there exists a real value, say $\bar{c}^*$, satisfying
\begin{multline*}
 \max\left\{\sup_{(\bar\beta,1]}\delta(\bar{f},\bar\beta),\ \bar{h}(\bar\beta)+2\sqrt{\dot{\bar{D}}(\bar\beta)\bar{g}(\bar\beta)}\right\} \le
 \\
  \bar{c}^* \le \sup_{(\bar\beta,1]}\delta(\bar{f},\bar\beta) + 2\sqrt{\sup_{\phi \in(\bar\beta,1]}\frac{1}{\phi-\bar\beta}\int_{\bar\beta}^{\phi}\frac{\bar{D}(s)\bar{g}(s)}{s}\,ds}
\end{multline*}
such that
 \begin{equation*}
 \left(\bar{D}(\phi_{1,\bar\beta})\phi_{1,\bar\beta}\rq{}\right)(\xi_{\bar\beta}^-)=
 \left\{
\begin{array}{ll}
0 \ &\mbox{ if } \ c\ge \bar{c}^*,\\
\ell <0 \ &\mbox{ if }\ c< \bar{c}^*,
\end{array}
\right.
 \end{equation*}
 and
 \begin{equation}
 \label{e:eq prime}
\phi_{1,\bar\beta}\rq{}(\xi_{\bar\beta}^-) =
\left\{
\begin{array}{ll}
\frac{\bar{g}(\bar\beta)}{\bar{s}_-(\bar\beta, c)} \ &\mbox{ if } \ c>\bar{c}^*,
\\[2mm]
\frac{\bar{g}(\bar\beta)}{\bar{s}_+\left(\bar\beta, \tilde{c}^*\right)} \ &\mbox{ if }\ c=\bar{c}^* \ \mbox{ and } \ \dot{\bar{D}}(\bar\beta)>0,
\\[2mm]
-\infty \ &\mbox{ if }\ c=\bar{c}^* \ \mbox{ and } \ \dot{\bar{D}}(\bar\beta)=0 \ \mbox{ or } \ c<\bar{c}^*.
\end{array}
\right.
\end{equation}
Here, $\bar{s}_\pm$ is defined as $s_\pm$  in \eqref{e:s} but with $\bar{D}$, $\bar{g}$ and $\bar{h}$ instead of the non-subscripted ones.
Define $c^*_{n,r}=\bar{c}^*$. The former of the two previous formulas yields to \eqref{e:threshold cm-dpde condition2}. Analogously, \eqref{e:eq prime} implies \eqref{e:phiprime betaplus}. Finally, {\em (i)}--{\em (iii)} follow from the application of {\em (i)--(iii)} of Proposition \ref{prop:cm-dpde} to $\phi_{1,\bar\beta}$ and \eqref{e:change 2}-\eqref{e:coefficients change}. The proof is then concluded.
\end{proof}

\smallskip\noindent

\begin{proofof}{Theorem \ref{thm:wf2}}
For each $c\in \R$, Propositions \ref{prop:cm-dpde}, \ref{prop:D1minus} provide the existence of a strict semi-wavefront $\phi_{1,\beta}$ from $1$, connecting $1$ to $\beta$, and  a strict semi-wavefront $\phi_{\beta,0}$ to $0$, connecting $\beta$ to $0$. Let $c^*_{np}$ be as in \eqref{e:c*np} and take $c\ge c^*_{np}$. Proposition \ref{prop:cm-dpde} informs us that $\phi_{1,\beta}$ satisfies $\eqref{e:threshold cm-dpde condition2}_1$ while Proposition \ref{prop:D1minus} implies that $\phi_{\beta,0}$ realizes  $\eqref{e:threshold cm-dpde condition}_1$. In addition, from the uniqueness up to shifts of both $\phi_{1,\beta}$ and $\phi_{\beta,0}$, we can suppose that their maximal-existence intervals are one the complement of the other and that the two of them have the finite extremum at the same $\xi_\beta \in \R$. Thus, by proceeding as in the proof of Theorem \ref{thm:wf}, we conclude that gluing together $\phi_{1,\beta}$ and $\phi_{\beta,0}$ at $\xi_\beta$ produces the desired wavefront.
This concludes the {\em if} part of the statement.

Suppose now that $\phi$ is a profile of a wavefront connecting $1$ to $0$ associated to some $c\in \R$. Necessarily, there exists a unique $\xi_\beta \in\R$ such that $\phi(\xi_\beta)=\beta$. Let $\phi_+$ be defined by $\phi_+(\xi)=\phi(\xi)$, for $\xi \in (-\infty, \xi_\beta)$. Here, the index \lq\lq{}+\rq\rq{} stands for the positive sign of $D\left(\phi_+\right)$ in its domain. The function $\phi_+$ is a semi-wavefront from $1$, connecting $1$ to $\beta$. Analogously, $\phi_-$ defined by $\phi_-(\xi)=\phi(\xi)$, for $\xi \in (\xi_\beta, +\infty)$ is a semi-wavefront to $0$, connecting $\beta$ to $0$. The function $\phi$ is a solution of Equation \eqref{e:ODE}, according to Definition \ref{d:tws}. Thus, if $\zeta\in C^\infty_0\left(-\infty,+\infty\right)$ is a test function with $\supp \zeta \subseteq \left[\xi_\beta-\delta, \xi_\beta+\delta\right]$ then we have
\begin{equation}
\label{equation c*np}
\int_{\xi_\beta-\delta}^{\xi_\beta+\delta} \left(D\left(\phi\right)\phi\rq{} -f\left(\phi\right) +c\phi \right)\zeta\rq{} -g\left(\phi\right)\zeta\,d\xi=0.
\end{equation}
Observe that both $\phi_+$ and $\phi_-$ are classical solution of \eqref{e:ODE} in $(-\infty,\xi_\beta)$ and $(\xi_\beta,+\infty)$, respectively, because in these domains $\pm D\left(\phi_{\pm}\right) >0$. Since $\int_{\xi_\beta-\delta}^{\xi_\beta+\delta}=\lim_{\delta>\eps\to 0^+}\int_{\xi_\beta-\delta}^{\xi_\beta-\eps}+\int_{\xi_\beta+\eps}^{\xi_\beta+\delta}$, from \eqref{equation c*np} and integration by parts we obtain
\begin{equation}\label{e:newnew}
\left(D\left(\phi_+\right)\phi_+\rq{}\right)\left(\xi_\beta^-\right)\zeta(\xi_\beta)=0 \ \mbox{ and } \ \left(D\left(\phi_-\right)\phi_-\rq{}\right)\left(\xi_\beta^+\right)\zeta(\xi_\beta)=0.
\end{equation}
Therefore, $\eqref{e:threshold cm-dpde condition}_1$ for $\phi_+$ and $\eqref{e:threshold cm-dpde condition2}_1$ for $\phi_-$ both hold true. As a consequence, by applying Proposition \ref{prop:cm-dpde} to $\phi_+$ and Proposition \ref{prop:D1minus} to $\phi_-$ we deduce that $c\ge c^*_{p,l}$ and $c\ge c^*_{n,r}$, from which $c\ge c^*_{np}$.
\par
Finally, the values of $\phi'(\xi_\beta^\pm)$ follow from each possible combination of \eqref{e:phiprime betaminus} applied to $\phi_{+}$ and \eqref{e:phiprime betaplus} applied to $\phi_-$. Analogously, {\em Parts (i)}--{\em (iii)} follow from putting together {\em (i)}--{\em (iii)} of Proposition \ref{prop:cm-dpde} applied to $\phi_+$ and {\em (i)}--{\em (iii)} of Proposition \ref{prop:D1minus} applied to $\phi_-$.
\end{proofof}

\begin{remark}
{
\rm
In Remark \ref{rem:c*pn} we deduced estimates from above and below for $c^*_{pn}$. Bounds for $c^*_{np}$ of Theorem \ref{thm:wf2} can be obtained similarly, starting now from \eqref{e:c*pl}, \eqref{e:c*nr} and \eqref{e:c*np}.
}
\end{remark}

\begin{remark}
\label{rem:reg at beta}
{
\rm
From \eqref{e:reg at beta}, we can make explicit the {\em a priori} regularity of $\phi$ at $\xi_\beta$ depending on the value of $c$ and the relative order between the thresholds (see also Figure \ref{f:f2s2}). We have:
\begin{enumerate}[(1)]
\item if $c>c^*_{np}$, we have $\phi\rq{}({\xi_\beta}^-)=\phi\rq{}({\xi_\beta}^+)\in (-\infty,0)$;

\item if $c=c^*_{np}=c^*_{n,r}=c^*_{p,l}$, we have $\phi\rq{}({\xi_\beta}^+)=\phi\rq{}({\xi_\beta}^-)\in (-\infty,0)$ if  $\dot{D}(\beta)>0$ and $\phi'(\xi_\beta^\pm)=
-\infty$ if $ \dot{D}(\beta)=0$; 

\item if $c=c^*_{np}=c_{n,r}^*>c^*_{p,l}$, we have $\phi\rq{}({\xi_\beta}^-) > \phi\rq{}({\xi_\beta}^+)$, with $\phi\rq{}({\xi_\beta}^-) \in (-\infty,0)$ and $\phi\rq{}({\xi_\beta}^+) \in [-\infty,0)$, with $\phi\rq{}({\xi_\beta}^+) =-\infty$ if and only if $\dot D(\beta)=0$;

\item if $c=c^*_{np}=c_{p,l}^*>c^*_{n,r}$, we have $\phi\rq{}({\xi_\beta}^+) > \phi\rq{}({\xi_\beta}^-)$, with $\phi\rq{}({\xi_\beta}^+) \in (-\infty,0)$ and $\phi\rq{}({\xi_\beta}^-) \in [-\infty,0)$, with $\phi\rq{}({\xi_\beta}^-) =-\infty$ if and only if $\dot D(\beta)=0$.
\end{enumerate}
}
\end{remark}

\section{Examples}
\label{sec:examples}
\setcounter{equation}{0}
In this section we provide some examples about Theorems \ref{thm:wf} and \ref{thm:wf2}.

\begin{example}
\label{ex:alpha infinity slope}
{
\rm
This example shows that $\eqref{e:phip(xialpha)}_2$ can indeed occur when a front goes from a positive- towards a negative-diffusivity region and the term $f$ is not identically zero. Consider $D$, $g$ and $f$ defined by
\begin{equation*}
D(\phi)=\begin{cases}
\phi\left(\phi-\frac1{2}\right)^2 \ &\mbox{ if } \ \phi \in [0,1/2],\\
-\left(1-\phi\right)\left(\frac12 -\phi\right)^2 \ &\mbox{ if } \ \phi \in (1/2,1],
\end{cases}
\
g(\phi)=\begin{cases}
\phi  \ &\mbox{ if } \ \phi \in [0,1/2],\\
1-\phi  \ &\mbox{ if } \ \phi \in  (1/2,1],
\end{cases}
\end{equation*}
and
\begin{equation*}
f(\phi)=\begin{cases}
\frac{\phi^3}{3}+\frac34 \phi^2 -\frac12 \phi \ &\mbox{ if } \ \phi \in [0,1/2],\\
\frac{\phi^3}{3}-\frac74\phi^2+2\phi-\frac{5}{8} \ &\mbox{ if } \ \phi \in  (1/2,1].
\end{cases}
\end{equation*}

Note, with these choices, $D$ satisfies ${\rm (D_{pn})}$ with $\alpha=1/2$ and $\dot{D}(\alpha)=0$ while (g) and (f) hold for $g$ and $f$. Moreover, $h(\alpha)=1/2$. From direct inspection, the function $z(\phi):=\phi\left(\phi-1/2\right)$, for $0\le \phi \le 1/2$, satisfies \eqref{e:zIntro}$_1$ with $c=0<h(\alpha)$. By integrating the formal identity $z(\phi)=D(\phi)\phi\rq{}$, the profile $\phi_{\alpha,0}$ of a semi-wavefront connecting $\alpha$ to $0$, with speed $c=0$, can be determined. In particular, the following Cauchy problem:
\[
\phi\rq{}=\frac1{\phi-1/2} \ \mbox{ and } \ \phi(0)=\frac1{4},
\]
is solved by $\phi_{\alpha,0}(\xi):= \frac12 -\sqrt{\frac1{16}+2\xi}$, for any $ \xi\ge -\frac1{32}$, and $\phi>0$ for $\xi < \frac3{32}$. Since $g(0)=0$, by setting $\phi_{\alpha,0}(\xi)=0$ for $\xi \ge \frac3{32}$ we have that $\phi_{\alpha,0}:(-\frac1{32},+\infty) \to [0,1/2)$ is the desired wave profile associated to the speed $c=0$. Moreover, with the notations of Theorem \ref{thm:wf}, $\xi_\alpha=-\frac1{32}$ and $\phi_{\alpha,0}\rq{}\left(\xi_\alpha^+\right)=-\infty$.
Similar arguments lead to conclude that
\begin{equation*}
\phi(\xi):=
\begin{cases}
1 \ &\mbox{ if } \ \xi \le -5/32,\\
\frac12 +\sqrt{-1/16 -2\xi} \ &\mbox{ if } \ -5/32 < \xi < -1/32,\\
\frac12 -\sqrt{1/16 +2\xi} \ &\mbox{ if } \ -1/32 \le\xi < 3/32, \\
0 \ &\mbox{ if } \ \xi \ge 3/32,
\end{cases}
\end{equation*}
is the profile of a (sharp at both $0$ and $1$) wavefront of \eqref{e:E} with speed $c=0$, such that $\phi'\left(\xi_\alpha^\pm\right)=-\infty$. Observe, from {\em Part (ii)} of Theorem \ref{thm:wf} it must occur $c^*_{pn}=0$.
}
\end{example}

\begin{example}
\label{ex:essentially different}
{
\rm
We show that $c^*_{pn}$ in Theorem \ref{thm:wf} and $c^*_{np}$ in Theorem \ref{thm:wf2} are {\em essentially} different: opposite diffusivities do {\em not} produce necessarily $c^*_{pn} = c^*_{np}$. To this aim, define
\begin{equation*}
\ g(\phi):=\left\{
\begin{array}{ll}
\phi^2 \ &\mbox{ if } \ \phi \in \left[0,1/2\right],\\[2mm]
\left(1-\phi\right)^2 \ &\mbox{ if }\ \phi\in\left(1/2,1\right],
\end{array}
\right.
\
f(\phi):=\left\{
\begin{array}{ll}
\phi^2\left(\phi -1\right) \ &\mbox{ if } \ \phi \in \left[0,1/2\right],\\[2mm]
\phi(1-\phi)^2 -1/4 \ &\mbox{ if }\ \phi\in\left(1/2,1\right].
\end{array}
\right.
\end{equation*}
The functions $g$ and $f$ satisfy (g) and (f). Let $D_1$ and $D_2$ be defined by
\begin{equation*}
D_1(\phi):=\left\{
\begin{array}{ll}
\phi\left(1/2-\phi\right) \ &\mbox{ if } \ \phi \in \left[0,1/2\right],\\[2mm]
-\left(1-\phi\right)\left(\phi-1/2\right) \ &\mbox{ if }\ \phi\in\left(1/2,1\right],
\end{array}
\right.
\ D_2(\phi)=-D_1(\phi) \ \mbox{ for } \ \phi \in [0,1].
\end{equation*}
With these choices, $D_1$ satisfies {$\rm (D_{pn})$} with $\alpha=1/2$ and $D_2$ satisfies {$\rm (D_{np})$} with $\beta=1/2$. From \eqref{e:c*pr} and $h(0)=\dot{f}(0)=0$ we have $c^*_{p,r}\ge 0$. Direct computations show that the function $z=z(\phi)$ defined by
\begin{equation}
\label{z1}
z(\phi):=\phi^2\left(\phi-1/2\right), \ \phi \in \left[0, 1/2\right],
\end{equation}
solves $\dot{z}=h-D_1g/z$ in $(0,1/2)$, $z<0$ in $(0,1/2)$ and $z(0)=0$. From the proof of \cite[Theorem 2.2]{BCM1} we have $c^*_{p,r}\le 0$. Then, $c^*_{p,r}=0$ follows at once by \eqref{e:c*pr}. Also, the symmetry of the coefficients implies that $c^*_{n,l}=c^*_{p,r}=0$. Thus,
\[
c^*_{pn}=0.
\]
Observe that we actually need to involve $z$ in \eqref{z1} since none of the intervals given by \eqref{e:c*pr} and \eqref{e:c*nl}, even in the sharper form involving \eqref{e:MPintro}, reduce to the point $\{0\}$.
\par
Starting from the formal identity $\dot{z}(\phi)=D(\phi)\phi\rq{}$, \eqref{z1}, we can compute the profile $\phi$ of the (unique up to shifts) wavefront associated to $c^*_{pn}$ in the current case. We have:
\[
\phi(\xi)=\begin{cases}
\frac14 e^{-\xi} \ &\xi>\log(1/2),\\[2mm]
1-e^{\xi} \ &\xi\le \log(1/2),
\end{cases}
\]
where in the interval $\xi \le \log(1/2)$ we make use of the symmetry of the problem.

On the other hand, consider now $D_2$, $g$ and $f$ together. We have
\[
c^*_{p,l} \ge \max\left\{\sup_{(\frac1{2},1]}\delta\left(f,1/2\right),\ h(1/2) + 2\sqrt{\dot{D_2}(1/2)g(1/2)}\right\} >0,
\]
since $h(1/2)=-1/4$ and $2\sqrt{\dot{D_2}(1/2)g(1/2)}=1/\sqrt{2}$. It follows necessarily that
\[
c^*_{np}\ge c^*_{p,l}>0.
\]
Then, we proved that replacing $D$, which satisfies $\rm (D_{pn})$, with $-D$, which then satisfies $\rm (D_{np})$, one can get $c_{np}^*>c_{pn}^*$.
}
\end{example}

\begin{example}
\label{ex:n-p gap}
{
\rm
In Theorem \ref{thm:wf2} the case $c^*_{p,l}\neq c^*_{n,r}$ reveals the existence of unusual {\em non-regular fronts}, where $\phi=\beta$, while in the case $c^*_{p,l}=c^*_{n,r}$ this is not possible. We now show an example in either case.

First, assume that $D$ and $g$ satisfy {$\rm (D_{np})$}-(g) and are such that $Dg$ is convex in $(0,\beta)$ and concave in $(\beta,1)$. For instance, we can take $\beta=\frac12$, $D(\phi)=\phi-1/2$ and $g(\phi)=\phi(1-\phi)$. We plainly have
\[
\sup_{\phi \in [0,\beta)}\frac1{\beta-\phi}\int_{\phi}^{\beta}\frac{(Dg)(s)}{s-\beta}\,ds=\dot{D}(\beta)g(\beta) \ \mbox{ and } \ \sup_{\phi \in (\beta,1]}\frac1{\phi-\beta}\int_{\beta}^{\phi}\frac{(Dg)(s)}{s-\beta}\,ds=\dot{D}(\beta)g(\beta).
\]
Suppose also that $f=0$ in $[0,1]$. Under these assumptions, inequalities \eqref{e:c*pl} and  \eqref{e:c*nr} become indeed equalities:
\[
c^*_{n,r} = c^*_{p,l}= c_{np}^* = 2\sqrt{\dot{D}(\beta)g(\beta)}.
\]
Second, consider $D\in C^1\left[0,1\right]$ such that $D<0$ in $(0,1/2)$ and $D(\phi)=\left(\phi-1/2\right)^3$, for $\phi \in [1/2,1]$; assume that $g$ satisfies (g) with $g(\phi)=1-\phi$, for $\phi \in [1/2,1]$, and let $f$ be defined by
\begin{equation*}
f(\phi):=\left\{
\begin{array}{ll}
0\ &\mbox{ if } \ \phi \in [0,\frac1{2}),\\
\left(\phi-\frac12\right)^2\left(\phi-\frac32\right)\ &\mbox{ if }\ \phi\in[\frac1{2},1].
\end{array}
\right.
\end{equation*}
For $1/2\le \phi \le 1$, set
\[
z(\phi):=\left(\phi-\frac12\right)^2\left(\phi-1\right).
\]
Such a $z$ solves \eqref{e:zIntro}$_2$ with $\alpha=1/2$ and $c=0$ in $(1/2,1)$. Moreover, since it also holds $z(1/2)=0$, then $0\ge c^*_{p,l}$ (see e.g. the proof of \cite[Theorem 2.2]{BCM1}). The left-hand side of \eqref{e:c*pl} implies that $c^*_{p,l}\ge 0$, because $h(1/2)=\dot{f}(1/2)=0$. Thus, $c^*_{p,l}=0$. On the other hand, if we have $h=0$ constantly in $(0,1/2)$, then $c^*_{n,r}>0$ (as already observed in Introduction, for the case of non-negative $D$ we refer to \cite[Remark 6.4]{BCM1} and reference therein; the case when $D$ is negative is treated similarly). Thus, in this case we have
\[
c^*_{np}=c^*_{n,r}>c^*_{p,l}.
\]
Similarly, one can provide examples where $c^*_{np}=c^*_{p,l}>c^*_{n,r}$.
}
\end{example}

\section{The case when $D$ vanishes more than once}
\label{sec:two-sign}
\setcounter{equation}{0}

In this final section we briefly outline how to cope with the case when $D$ changes sign more than once, focusing on the simplest of these cases. More precisely we consider one of the following assumptions, see Figure \ref{f:DD}:
\begin{itemize}
\item[{$(\rm D_{pnp})$}] \, $D\in C^1[0,1]$, $D>0$ in $(0,\alpha)\cup (\beta,1)$ and $D<0$ in $(\alpha,\beta)$, with $\alpha<\beta$;

\item[{$(\rm D_{npn})$}] \, $D\in C^1[0,1]$, $D<0$ in $(0,\beta)\cup (\alpha,1)$ and $D>0$ in $(\beta, \alpha)$, with $\beta<\alpha$.
\end{itemize}

\begin{figure}[htb]
\begin{center}

\begin{tikzpicture}[>=stealth, scale=0.5]

\draw[->] (0,0) --  (6,0) node[below]{$\rho$} coordinate (x axis);
\draw[->] (0,0) -- (0,3) node[right]{$D$} coordinate (y axis);
\draw (0,0) -- (0,-1.5);
\draw[thick] (0,1) .. controls (0.7,2) and (1.3,2) .. (1.7,0) node[left=4, below=0]{\footnotesize{$\alpha$}};
\draw[thick] (1.7,0) .. controls (2.1,-2) and (3,-2) .. (3.5,0);
\draw[thick] (3.5,0) node[right=3, below=0]{\footnotesize{$\beta$}} .. controls (4,2) and (4.5,2) .. (5,1);
\draw[dotted] (5,1) -- (5,0) node[below]{\footnotesize{$1$}};
\draw(3,-3) node[above]{$(\rm D_{pnp})$};

\begin{scope}[xshift=10cm]
\draw[->] (0,0) --  (6,0) node[below]{$\rho$} coordinate (x axis);
\draw[->] (0,0) -- (0,3) node[right]{$D$} coordinate (y axis);
\draw (0,0) -- (0,-1.5);
\draw[thick] (0,-1) .. controls (0.7,-2) and (1.3,-2) .. (1.7,0) node[right=4, below=0]{\footnotesize{$\beta$}};
\draw[thick] (1.7,0) .. controls (2.1,2) and (3,2) .. (3.5,0);
\draw[thick] (3.5,0) node[left=3, below=0]{\footnotesize{$\alpha$}} .. controls (4,-2) and (4.5,-2) .. (5,-1);
\draw[dotted] (5,-1) -- (5,0) node[above]{\footnotesize{$1$}};
\draw(3,-3) node[above]{$(\rm D_{npn})$};
\end{scope}

\end{tikzpicture}

\end{center}
\caption{\label{f:DD}{Typical plots of the functions $D$.}}
\end{figure}
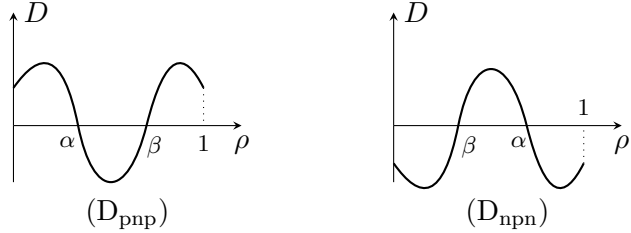

About the case $(\rm D_{pnp})$ we refer to \cite{Ferracuti-Marcelli-Papalini, Kuzmin-Ruggerini}, where $g$ is, respectively, monostable and bistable; \cite{Bao-Zhou2017} for the case {$(\rm D_{pnp})$} and monostable $g$, but with a specific quadratic diffusivity $D$. The case $(\rm D_{npn})$ has never been considered.

\begin{theorem}
\label{thm:pnp}
Assume {\rm (f)}, {$\rm (D_{pnp})$}, {\rm (g)} and $\eqref{Dg}_1$. Then, there exists $c^*_{pnp}\in \R$ such that Equation \eqref{e:E} admits a (unique up to space shifts) wavefront, with speed $c$ and profile $\phi$ satisfying \eqref{e:infty}, if and only if $c \ge c^*_{pnp}$.
\end{theorem}

\begin{theorem}
\label{thm:npn}
Assume {\rm (f)}, {$\rm (D_{npn})$}, {\rm (g)} and $\eqref{Dg}_2$. Then, there exists $c^*_{npn}\in \R$ such that Equation \eqref{e:E} admits a (unique up to space shifts) wavefront, with speed $c$ and profile $\phi$ satisfying \eqref{e:infty}, if and only if $c \ge c^*_{npn}$.
\end{theorem}

In both theorems one can easily deduce more informations on the profiles, as in Theorems \ref{thm:wf} and \ref{thm:wf2}. We leave these details to the reader.

\begin{proofof}{Theorem \ref{thm:pnp}}
We divide our problem in three sub-problems corresponding to the three connected components of $\left\{D\neq 0\right\}$. To the intervals $(0,\alpha)$ and $(\beta,1)$ we apply Propositions \ref{p:swf to zero} and \ref{prop:cm-dpde}, respectively. Indeed, the results of Sections \ref{sec:dpn} and \ref{sec:dnp} hold under lighter assumptions, involving only that part of the conditions corresponding to the interval under consideration.

Let ${c}_{p,r}^*$ and ${c}_{p,l}^*$ be the thresholds appearing in Propositions \ref{p:swf to zero} and \ref{prop:cm-dpde}, respectively. Hence, associated to the same speed $c$, both a semi-wavefront connecting $\alpha$ and $0$, with profile $\phi_{\alpha,0}$ satisfying \eqref{e:xialpha phi} and a semi-wavefront connecting $\beta$ to $1$, whose profile is $\phi_{1,\beta}$ satisfying $\eqref{e:threshold cm-dpde condition}_1$, are given, if and only if $c\ge \max \{{c}^*_{p,r}, {c}_{p,l}^*\}$.
\par
We claim that there exists $c^*_{\alpha\beta}\in\R$ such that the following holds: if and only if $c\ge c^*_{\alpha\beta}$, Equation \eqref{e:E} admits a strict TW, connecting $\beta$ to $\alpha$, whose non-increasing profile $\phi_{\beta,\alpha}\in (\alpha, \beta)$ is defined in some interval $(\xi_\beta, \xi_\alpha)$, where $-\infty < \xi_\beta < \xi_\alpha < +\infty$ are such that $\phi_{\beta,\alpha}(\xi_\beta^+)=\beta \ \mbox{ and } \ \phi_{\beta,\alpha}(\xi_\alpha^-)=\alpha$ and also
\begin{equation*}
\label{e:e1}
\left(D\left(\phi_{\beta,\alpha}\right)\phi_{\beta,\alpha}\rq{}\right)\left(\xi_\beta^+\right) = 0 = \left(D\left(\phi_{\beta,\alpha}\right)\phi_{\beta,\alpha}\rq{}\right)\left(\xi_\alpha^-\right).
\end{equation*}
Then by gluing together $\phi_{1,\beta}$, $\phi_{\beta, \alpha}$ and $\phi_{\alpha,0}$ (modulo shifts) we obtain the front of \eqref{e:E} satisfying \eqref{e:infty}, in virtue of \eqref{e:e1}, as in the proof of Theorem \ref{thm:wf}. Clearly, we define
\[
c^*_{pnp}=\max\left\{ c^*_{p,r}, c_{p,l}^*, c^*_{\alpha \beta}\right\}.
\]
This proves the {\em if} part of Theorem \ref{thm:pnp}.
Viceversa, assume that a wavefront $\phi$ of \eqref{e:E}, associated to some speed $c$, is given. As in the proof of the {\em only if} parts of Theorems \ref{thm:wf} and \ref{thm:wf2}, $\phi$ is decomposed into a semi-wavefront which connects $1$ to $\beta$ whose profile satisfies \eqref{e:threshold cm-dpde condition}, a strict TW connecting $\beta$ to $\alpha$ and a strict semi-wavefront connecting $\alpha$ to $0$. By Propositions \ref{p:swf to zero} and \ref{prop:swf from one}, the {\em only if} part of our claim implies $c\ge c^*_{pnp}$. Thus, Theorem \ref{thm:pnp} follows from the claim. We refer to Figure \ref{f:3phi}.

\begin{figure}[htb]
\begin{center}

\begin{tikzpicture}[>=stealth, scale=0.7]

\draw[->] (0,0) --  (10,0) node[below]{$\xi$} coordinate (x axis);
\draw[->] (4,0) -- (4,4) node[right]{$\phi$} coordinate (y axis);
\draw (0,3) -- (10,3);
\draw(4,3.3) node[left]{\footnotesize{$1$}};
\draw[dotted] (4,2.4) node[right=5,above=-2]{\footnotesize{$\beta$}} -- (10,2.4);
\draw[dotted] (0,2.4) -- (4,2.4);
\draw[dotted] (3.5,0) node[below]{$\xi_\beta$} -- (3.5,2.4);
\draw[dotted] (4,0.8) node[right=5,below=-2]{\footnotesize{$\alpha$}} -- (10,0.8);
\draw[dotted] (0,0.8) -- (4,0.8);
\draw[dotted] (5.9,0) node[below]{$\xi_\alpha$} -- (5.9,0.8);
\draw[thick] (0,2.8) .. controls (7,2.8) and (3,0.2) .. node[below,very near start]{$\phi_{1,\beta}$} node[right,midway]{$\phi_{\alpha,\beta}$} node[above,very near end]{$\phi_{\alpha,0}$} (10,0.1); 
\end{tikzpicture}

\end{center}
\caption{\label{f:3phi}{The pasting of the profiles.}}
\end{figure}
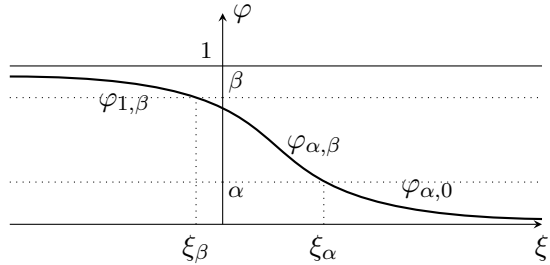

We prove the claim. Let $\bar D$, $\bar g$, $\bar f$ and $\bar \alpha$ be as in \eqref{e:coefficients change} and let $\bar h$ be the derivative of $\bar f$, which gives  $\bar h(\phi)= h(1-\phi)$. Set $\bar \beta=1-\beta$. Let $c$ be a real value and consider the following equation in the unknown $\psi=\psi(\xi) \in [\bar \beta, \bar \alpha]$:
\begin{equation}
\label{e:psi} \left(\bar D(\psi) \psi\rq{}\right)\rq{} +\left(c- \bar h(\psi)\right)\psi\rq{} +\bar g(\psi)=0.\end{equation}
We are interested in a non-increasing $\psi$ satisfying \eqref{e:psi} in some interval $(\xi_{\bar\alpha}, \xi_{\bar\beta})$, with $-\infty < \xi_{\bar\alpha} < \xi_{\bar \beta} < +\infty$, such that $\psi({\xi_{\bar \alpha}}^+)=\bar \alpha$, $\psi({\xi_{\bar\beta}}^-)=\bar\beta$ and 
\[ \left(\bar D(\psi)\psi\rq{}\right)({\xi_{\bar\alpha}}^+)=0=\left(\bar D(\psi)\psi\rq{}\right)({\xi_{\bar\beta}}^-).\]

By using a standard argument (see e.g. \cite{MMconv}) based on the monotonicity of $\psi$ and the definition of $z(\psi):=\bar D(\psi)\psi\rq{}$, the existence of a $\psi$ as above, but with not necessarily finite $\xi_{\bar\alpha}$ and $\xi_{\bar\beta}$, is equivalent to the solvability of the problem
\[ \begin{cases} \dot z(\psi)= \bar h(\psi)-c- \frac{\bar D (\psi)\bar g(\psi)}{z(\psi)}, \ &\psi \in (\bar\beta, \bar\alpha),\\
z(\psi)<0 \ &\psi \in (\bar\beta, \bar\alpha),\\
z(\bar\beta)=z(\bar \alpha)=0.\end{cases}\]
Since $\bar D >0$ in $(\bar\beta, \bar\alpha)$, $\bar D(\bar\beta)=\bar D (\bar\alpha)=0$, $\bar g >0 $ in $[\bar \beta, \bar \alpha]$, by applying \cite[Proposition 4.2]{BCM1}, such a problem is solvable if and only if $c \ge c^*$, for some $c^* \in \R$. Also, since (by arguing as in \cite[Lemma 9.1]{BCM1} and \cite[proof of Theorem 2.5]{CM-DPDE} where analogous assumptions were considered) we have
\[ \lim_{\psi \to \bar\alpha^-}\frac{\bar D (\psi)}{z(\psi)} \in \R \ \mbox{ and } \ \lim_{\psi \to \bar \beta ^+}\frac{\bar D (\psi)}{z(\psi)} \in \R,\]
then $\xi_{\bar\alpha}$ and $\xi_{\bar\beta}$ are finite because $\psi$ is bounded and $\psi\rq{}(\xi) <0$ for $\xi$ which tends to either ${\xi_{\bar\alpha}}^+$ or ${\xi_{\bar\beta}}^-$. Finally, setting $c_{\alpha\beta}^* :=c^*$ concludes the proof of the claim since the desired $\phi_{\beta,\alpha}$ exists if and only if there exists $\psi$, by $\phi_{\beta, \alpha}(\xi):=1-\psi(-\xi)$.

\end{proofof}

\begin{remark}
\label{rem:final1}
{
\rm
The analog of Theorem \ref{thm:pnp} when $f=0$ was given in the first part of \cite[Theorem 4.2]{Ferracuti-Marcelli-Papalini} with slightly different notation. We point out that in \cite{Ferracuti-Marcelli-Papalini} wavefronts are necessarily classical at $1$ (called classical or sharp of type (I), there). Instead, when $D(1)=0$ and $c<h(1)$  the wavefronts of Theorem \ref{thm:pnp} can be sharp at $1$, too, by applying Part $(iii)$ of Proposition \ref{prop:cm-dpde} to $\phi_{1,\beta}$ of the proof of Theorem \ref{thm:pnp}. Note, $c<h(1)$ cannot happen if $f=0$. 

We also observe that, as in \cite[Theorem 4.2]{Ferracuti-Marcelli-Papalini}, wavefronts are sharp at $0$ only if $D(0)=0$ and $c=c^*_{pnp}$.  Moreover, if $D(0)=0$ and  $c=c^*_{pnp}$, generalizing \cite[(4.1)]{Ferracuti-Marcelli-Papalini}, it is possible to provide a sufficient condition on $f$ and $Dg$ in order to make wavefronts of  Theorem \ref{thm:pnp} always classical at $0$.
}
\end{remark}

\begin{proofof}{Theorem \ref{thm:npn}}
We proceed in the spirit of the proof of Theorem \ref{thm:pnp}. We consider separately the intervals where $D$ has constant sign. From Proposition \ref{prop:swf from one} we deduce that a semi-wavefront connecting $1$ to $\alpha$, with speed $c$ and profile satisfying \eqref{e:ell2} with $s=0$, exists if and only if $c \ge c^*_{n,l}$. Proposition \ref{prop:D1minus} implies that there exists a semi-wavefront connecting $\beta$ to $0$ with speed $c$ and profile satisfying $\eqref{e:threshold cm-dpde condition2}_1$ if and only if $c\ge c^*_{n,r}$. About the interval $(\beta,\alpha)$, we argue as in the proof of Theorem \ref{thm:pnp}, but considering Equation \eqref{e:ODE}, directly, rather than Equation \eqref{e:psi}. We have that a strict TW connecting $\alpha$ to $\beta$ with speed $c$ and profile $\phi_{\alpha,\beta}:(\xi_\alpha, \xi_\beta) \to (\beta,\alpha)$ satisfying $\phi_{\alpha,\beta}(\xi_\alpha^+)=\alpha$, $\phi_{\alpha,\beta}(\xi_\beta^-)=\beta$ and such that $\left(D(\phi_{\alpha,\beta})\phi_{\alpha,\beta}\right)(\xi)$ tends to $0$ if either $\xi\to \xi_\alpha^+$ and $\xi\to \xi_\beta^-$, exists if and only if $c\ge c^*_{\beta\alpha}$. To conclude the proof, we set
$c^*_{npn}:=\max\{c_{n,l}^*,c^*_{n,r},c^*_{\beta\alpha}\}$.
\end{proofof}

\begin{remark}
\label{rem:final2}
{
\rm
Similarly to Remark \ref{rem:final1}, we can deduce information on the regularity of wavefronts given in Theorem \ref{thm:npn}.
}
\end{remark}

\section*{Acknowledgments}
The authors are members of the {\em Gruppo Nazionale per l'Analisi Matematica, la Probabilit\`{a} e le loro Applicazioni} (GNAMPA) of the {\em Istituto Nazionale di Alta Matematica} (INdAM) and acknowledge financial support from this institution.

%

{\small
\bibliography{refe_BCM2}
\bibliographystyle{abbrv2}
}

\end{document}